\documentclass[a4paper,12pt,intlimits,oneside]{amsart}
\usepackage{amsmath}
\usepackage{amsthm}
\usepackage{latexsym}
\usepackage{amssymb}
\usepackage{xcolor}
\usepackage{graphicx}
\usepackage[colorlinks=true]{hyperref}
\numberwithin{figure}{section}
\def\R{{\mathbb R}}
\def\C{{\mathbb C}}

\def\T{{\mathbb T}}
\def\Z{{\mathbb Z}}

\def\N{{\mathbb N}}
\def\la{\langle}
\def\ra{\rangle}

\def\e{\varepsilon}

\def\build#1_#2^#3{\mathrel{
\mathop{\kern 0pt#1}\limits_{#2}^{#3}}}
\def\td_#1,#2{\mathrel{\mathop{\build\longrightarrow_{#1\rightarrow #2}^{}}}}
\DeclareFontFamily{U}{MnSymbolC}{}
\DeclareSymbolFont{MnSyC}{U}{MnSymbolC}{m}{n}
\DeclareFontShape{U}{MnSymbolC}{m}{n}{
    <-6>  MnSymbolC5
   <6-7>  MnSymbolC6
   <7-8>  MnSymbolC7
   <8-9>  MnSymbolC8
   <9-10> MnSymbolC9
  <10-12> MnSymbolC10
  <12->   MnSymbolC12}{}
\DeclareMathSymbol{\intprod}{\mathbin}{MnSyC}{'270}
\newtheorem{theorem}{Theorem}
\newtheorem{corollary}{Corollary}
\newtheorem{proposition}{Proposition}
\newtheorem{lemma}{Lemma}
\newtheorem{remark}{Remark}
\newtheorem{definition}{Definition}
\begin{document}
\title[Sharp well-posedness results of the BO equation]{Sharp well-posedness results of the Benjamin-Ono equation in $H^{s}(\T, \R)$ and
qualitative properties of its solution}
\author[P. G\'erard]{Patrick G\'erard}
\address{Laboratoire de Math\'ematiques d'Orsay, CNRS, Universit\'e Paris--Saclay, 91405 Orsay, France} \email{{\tt patrick.gerard@math.u-psud.fr}}
\author[T. Kappeler]{Thomas Kappeler}
\address{Institut f\"ur Mathematik, Universit\"at Z\"urich, Winterthurerstrasse 190, 8057 Zurich, Switzerland} 
\email{{\tt thomas.kappeler@math.uzh.ch}}
\author[P. Topalov]{Petar Topalov}
\address{Department of Mathematics, Northeastern University,
567 LA (Lake Hall), Boston, MA 0215, USA}
\email{{\tt p.topalov@northeastern.edu}}


\date{\textcolor{red}{March 25,  2020}}

\begin{abstract}
We prove that the Benjamin--Ono equation on the torus
is globally in time well-posed in 
the Sobolev space $H^{s}(\T, \R)$ for any $s > - 1/2$ and ill-posed for $s \le - 1/2$. 
Hence the critical Sobolev exponent $s_c=-1/2$
of the Benjamin--Ono equation is the threshold 
for well-posedness on the torus. 
The obtained solutions are almost periodic in time. 
Furthermore, we prove that  the traveling wave solutions of the Benjamin--Ono equation 
on the torus are orbitally stable in $H^{s}(\T ,\R)$ for any $ s > - 1/2$. 
Novel conservation laws and a nonlinear
Fourier transform on $H^{s}(\T, \R)$ with $s > - 1/2$
are key ingredients into the proofs of these results.
\end{abstract}


\thanks{We would like to warmly thank J.C. Saut for
very valuable discussions and for
making us aware of many references, in particular
\cite{AN}. We also thank T. Oh for bringing reference 
\cite{AH} to our attention. 
T.K. is partially supported by the Swiss National Science Foundation.
P.T. is partially supported by the Simons Foundation, Award \#526907.}

\maketitle

\tableofcontents

\medskip

\section{Introduction}\label{Introduction}

In this paper we consider the Benjamin-Ono (BO) equation on the torus,
\begin{equation}\label{BO}
\partial_t v = H\partial^2_x v - \partial_x (v^2)\,, \qquad x \in \T:= \R/2\pi\Z\,, \,\,\, t \in \R,
\end{equation}
where $v\equiv v(t, x)$ is real valued and $H$ denotes the Hilbert transform, defined for $f = \sum_{n \in \mathbb Z} \widehat f(n) e^{ i n x} $,
$\widehat f(n) = \frac{1}{2\pi}\int_0^{2\pi} f(x) e^{-  i n x} dx$,  by
$$
H f(x) := \sum_{n \in \Z} -i \ \text{sign}(n) \widehat f(n) \ e^{inx}
$$
with $\text{sign}(\pm n):= \pm 1$ for any $n \ge 1$, whereas $\text{sign}(0) := 0$. This pseudo-differential equation ($\Psi$DE)
in one space dimension has been introduced by Benjamin \cite{Benj} and Ono \cite{Ono} to model long, uni-directional internal gravity waves 
in a two-layer fluid. It has been extensively studied, 
both on the real line $\R$ and on the torus $\T$. 
For an excellent survey, including the derivation of  \eqref{BO}, we refer to the recent article by Saut \cite{Sa}.

Our aim is to study  low regularity solutions of the BO equation on $\T$. To state our results, we first need to review
some classical results on the well-posedness problem of \eqref{BO}.
Based on work of Saut \cite{Sa0}, 
Abdelouhab, Bona, Felland, and Saut
proved in \cite{ABFS} that
for any $s \ge 3/2$, equation \eqref{BO} is globally in time 
well-posed on the Sobolev space $H^s_r \equiv H^s(\T, \R) $ 
(endowed with the standard norm $\| \cdot \|_s$, defined by \eqref{Hs norm} below), meaning the following:
\begin{itemize}
\item[(S1)]
{\em Existence and uniqueness of classical solutions:} For any initial data $v_0 \in H^{s}_r$, there exists 
a unique curve $v : \R \to H^s_r$ in
$C(\R, H^s_r) \cap C^1(\R, H^{s-2}_r)$ so that
$v(0) = v_0$ and for any $t \in \R$, equation \eqref{BO} is satisfied in $H^{s-2}_r$.   
(Since $H^s_r$ is an algebra, one has 
$\partial_xv(t)^2 \in H^{s-1}_r$ for any time $t \in \R$.)
\item[(S2)]{\em Continuity of solution map:}
The solution map 
$\mathcal S : H^s_r \to C(\mathbb R, H^s_r)$
is continuous, meaning that for any $v_0 \in H^s_r,$
$T > 0$, and $\varepsilon > 0$ there exists $\delta > 0,$ 
so that for any $w_0 \in H^s_r$ with $\|w_0 - v_0 \|_s < \delta$,
the solutions $w(t) = \mathcal S(t, w_0)$ and 
$v(t)= \mathcal S(t, v_0)$ of \eqref{BO}
with initial data $w(0) = w_0$ and, respectively, $v(0) = v_0$
satisfy $\sup_{|t| \le T} \| w(t) - v(t) \|_s \le  \varepsilon$.
\end{itemize}
In a straightforward way one verifies that
\begin{equation}\label{CL}
 \mathcal H^{(-1)}(v) := \langle v | 1 \rangle  \, , \qquad
 \mathcal H^{(0)}(v) :=  \frac{1}{2} \langle v | v \rangle  
\end{equation}
are integrals of the above solutions of \eqref{BO}.
Here $\langle \cdot  \, |  \, \cdot \rangle $ denotes the $L^2-$inner product, 
\begin{equation}\label{L2 inner product}
\langle f | g \rangle =  \frac{1}{2\pi} \int_0^{2\pi} f \overline g dx \, .
\end{equation}
In particular it follows that for any $c \in \R$ 
and any $s \ge 3/2$,
the affine space $H^s_{r,c}$ is left invariant by $\mathcal S$
where for any $\sigma \in \R$
\begin{equation}\label{affine spaces}
H^\sigma_{r,c} :=  
\{ w \in H^\sigma_r \, : \, \langle w | 1 \rangle =c \}\, .
\end{equation}

\noindent
In the sequel, further progress has been made 
on the well-posedness of \eqref{BO} 
on Sobolev spaces of low regularity. 
The best results so far in this direction
were obtained by Molinet  by using the gauge transformation introduced by Tao \cite{Tao}. 
Molinet's results in \cite{Mol} ( cf. also \cite{MP}) imply that the solution 
map $\mathcal S,$ introduced in $(S2)$ above, 
continuously extends to any
Sobolev space $H^s_r$ with $0 \le s \le 3/2$. 
More precisely, for any such $s$,
$\mathcal S: H^s_r \to C(\R, H^s_r)$ is continuous
and for any $v_0 \in H^s_r$, $ \mathcal S(t, v_0)$ satisfies equation \eqref{BO}  
in $H^{s-2}_r$. The fact that $\mathcal S$ continuously
extends to $L^2_{r} \equiv H^0_{r}$, 
$\mathcal S : L^2_{r} \to C(\R, L^2_{r})$,
 can also be deduced by methods
recently developed in \cite{GK}. Furthermore, one infers from \cite{GK}
that any solution  $ \mathcal S(t, v_0)$ with initial data $v_0 \in L^2_r$
can be approximated in $C(\R, L^2_r)$ by  solutions of \eqref{BO} which are rational functions of $\cos x$, $\sin x$.
We refer to these solutions as rational solutions.

In this paper we show that the BO equation is well-posed in the Sobolev space
$H^{-s}_r$ for any $0 < s < 1/2$ and that this result is sharp. 
Since the nonlinear term $\partial_x v^2$ in equation \eqref{BO} is not 
well-defined for elements in $H^{-s}_r$, we first need to define what
we mean by a solution of \eqref{BO} in such a space.

\begin{definition}\label{def solution}
Let $s \ge 0$.
A continuous curve $\gamma: \R \to H^{-s}_r$ 
with $\gamma(0) = v_0$ for a given $v_0 \in H^{-s}_r$,
is called a global in time solution of the BO equation in 
$H^{-s}_r$ with initial data $v_0$ if for any
sequence $(v_0^{(k)})_{k \ge 1}$ in $H^\sigma_r$ with $\sigma > 3/2,$
which converges to $v_0$ in $H^{-s}_r$, the corresponding
sequence of classical solutions $\mathcal S(\cdot, v_0^{(k)})$
converges to $\gamma$ in $C(\R, H^{-s}_{r})$.
The solution $\gamma$ is denoted by 
$\mathcal S(\cdot, v_0)$.
\end{definition}
\noindent
We remark that for any $v_0 \in L^2_r$, the solution
$\mathcal S(\cdot, v_0)$ in the sense of Definition \ref{def solution}
coincides with the solution obtained by Molinet in \cite{Mol}.
\begin{definition}\label{def well-posedness}
Let $s \ge 0$.
Equation \eqref{BO} is said to be globally $C^0-$well-posed
in $H^{-s}_r$ if the following holds:
\begin{itemize}
\item[(i)]
For any $v_0 \in H^{-s}_r$, there exists
a global in time solution of \eqref{BO} 
with initial data $v_0$
in the sense of Definition \ref{def solution}.
\item[(ii)]
The solution map $\mathcal S : H^{-s}_r \to C(\R, H^{-s}_r)$
is continuous, i.e. satisfies (S2).
\end{itemize}
\end{definition}
 Our main results are the following ones:
\begin{theorem}\label{Theorem 1}
For any $0 \le s < 1/2$,  the Benjamin-Ono equation is globally $C^0-$well-posed
on $H^{-s}_r$ in the sense of Definition \ref{def well-posedness}.  For any $c \in \R$,  $t\in \R $,  the flow map $ S^t=\mathcal S(t, \cdot)$  leaves
the affine space $H^{-s}_{r,c}$,  introduced in \eqref{affine spaces}, invariant.\\
Furthermore, there exists a conservation law
$ I_{-s} : H^{-s}_r \to \R_{\ge 0}$ of \eqref{BO} 
satisfying
$$
\| v \|_{-s} \le I_{-s}(v)\, , \quad \forall v \in H^{-s}_r \, .
$$
In particular, one has
$$
\sup_{t \in \R} \|\mathcal S(t, v_0)\|_{-s}
\le I_{-s} (v_0)\ , \quad \forall v_0 \in H^{-s}_r\, .
$$
\end{theorem}
\begin{remark}
$(i)$ Theorem \ref{Theorem 1} continues to hold on $H^{s}_r$ for any $s> 0$. See Corollary \ref{results for s positive} in Appendix \ref{restrictions}.\\
$(ii)$ Since by \eqref{CL}, the $L^2-$norm is an integral of \eqref{BO}, $\textcolor{red}{I_{-s}}$ in the case $s=0$
can be chosen as $I_0(v):= \| v \|_0^2$. The
definition of $I_{-s}$ for $0 < s < 1/2$ can be found in Remark \ref{def integral I_s} in  Section \ref{Birkhoff map}.
These novel integrals are one of the key ingredients for the proof of global $C^0-$well-posedness of \eqref{BO} in $H^{-s}_r$ for $0 < s < 1/2$.
\\
$(iii)$ Note that global $C^0$--well-posedness implies the group property $S^{t_1}\circ S^{t_2}=S^{t_1+t_2}$. Consequently, $S^t$ is a homeomorphism of $H^{-s}_r$. \\
$(iv)$ By Rellich's compactness theorem, $S^t $ is also  weakly sequentially continuous on $H^{-s}_{r, c}$ for any $0 \le s < 1/2$ and $c \in \mathbb R$, hence in particular on $L^2_{r,0}$. 
Note that this  contradicts a result stated in \cite[Theorem 1.1]{Mol1}. Very recently, however, 
an error in the proof of the latter theorem has been found,
leading to the withdrawal of the paper (cf. arXiv:0811.0505). 
A proof of this weak continuity property was indeed the starting point 
of the present paper.
\end{remark}
The next result says that the well-posedness result of Theorem \ref{Theorem 1} is sharp.
\begin{theorem}\label{Theorem 2}
For any $c \in \mathbb R,$ the Benjamin-Ono equation is ill-posed on $H^{-1/2}_{r,c}$.
More precisely,
there exists a sequence $(u^{(k)})_{k \ge 1}$  in $\bigcap_{n \ge 1}H^{n}_{r,0}$, converging strongly to $0$ in $H^{-1/2}_{r,0}$,
so that  for any $c \in \mathbb R$, the solutions  $\mathcal S(t, u^{(k)} + c)$ of \eqref{BO}
of average $c$ have 
the property that the sequence of functions $t\mapsto \la \mathcal S(t, u^{(k)} + c) \vert \ {\rm e}^{ix} \ra $ 
does not converge pointwise to $0$ on any given time interval of positive length. 
\end{theorem}
\begin{remark}\label{illposedness}  
It was observed in \cite{AH} that
the solution map $\mathcal S$ does not continuously extend
to $H^{-s}_{r}$ with $s > 1/2$. More precisely, for any $c \in \R$,
the authors of \cite{AH} construct a sequence
$(v_0^{(k)})_{k \ge 1}$ in $\bigcap_{n \ge 0} H^n_{r,c}$ of initial data
so that for any $s > 1/2$ it converges to an element $v_0$ in $H^{-s}_{r,c}$
whereas for any $t \ne 0$,  $(\mathcal S(t, v_0^{(k)}))_{k \ge 1}$ diverges even in the sense of distributions. 
However, the divergence of  $\mathcal S(t, v_0^{(k)})$ can be removed by
renormalizing the flow by a translation of the space variable, $x \mapsto x+ \eta_kt$. In the case $c=0$, $\eta_k$ is given 
by $\|v_0^{(k)}\|_0^2$.  We refer to \cite{C} for a similar renormalization in the context of the nonlinear Schr\"odinger equation. 
In Appendix \ref{illposedness for s < -1/2}, we construct a sequence of initial data in $\bigcap_{n \ge 0} H^n_{r,c}$ 
with the above convergence/divergence properties, but where such a renormalization is {\em not} possible.
\end{remark}

\smallskip
\noindent
{\bf Comments on Theorem \ref{Theorem 1} and Theorem \ref{Theorem 2}.} 
$(i)$ A straightforward computation shows that $s_c = - 1/2$ is
the critical Sobolev exponent of the Benjamin-Ono equation.
Hence Theorem \ref{Theorem 1} and Theorem \ref{Theorem 2} say
 that the threshold of well-posedness of \eqref{BO}  is 
given by the critical Sobolev exponent $s_c$.\\
$(ii)$  In a recent, very interesting paper \cite{Tal}, 
Talbut proved by the method of perturbation determinants,
developed for the KdV and the NLS equations by Killip, Visan, and Zhang in \cite{KVZ},
that for any $0 < s < 1/2,$ there exists a constant 
$C_s>0$, only depending on $s$,
so that any sufficiently smooth solution $t \to v(t)$ 
of \eqref{BO} satisfies the estimate
$$
\sup_{t\in \R}\| v(t)\|_{-s}\leq  C_s 
\big( 1 + \|v(0)\|_{-s}^{\frac{2}{1-2s}} \big)^s
\|v(0)\|_{-s} \, .
$$
We note that the integrals $I_{-s}$ of Theorem \ref{Theorem 1}$(iv)$
are of a different nature. Let us explain this in more detail.
Our method for proving that the solution map $\mathcal S$ of \eqref{BO} 
continuously extends to $H^{-s}_r$ for any $0 < s < 1/2$
consists in constructing a globally defined nonlinear Fourier transform $\Phi$,
also referred to as Birkhoff map (cf. Section \ref{Birkhoff map}). It means 
that \eqref{BO} can be solved by quadrature, when expressed in the coordinates
defined by $\Phi$, which we refer to as Birkhoff coordinates.
The integrals $I_{-s}$ of Theorem \ref{Theorem 1}$(iv)$ are taylored to show 
that $\Phi: H^{-s}_{r,0} \to h^{1/2 -s}_+$ is onto  (cf. Theorem \ref{extension Phi}). 
Actually, the map $\Phi$ is a key ingredient not only in the proof of Theorem \ref{Theorem 1},
but in the proof of all results, stated in Section \ref{Introduction}.
In particular, with regard to Theorem \ref{Theorem 2}, we note that
the standard norm inflation argument, pioneered by  \cite{CCT}
 (cf. also \cite[Appendix A]{HKV} and references therein),
does not apply for proving ill-posedness of \eqref{BO} in $H^{-1/2}_{r,c}$
since the mean $\langle u | 1 \rangle$ is an integral of \eqref{BO}.
Our proof of Theorem \ref{Theorem 2} is based in a fundamental way on the map $\Phi$
and its properties (cf. Section \ref{Proof of Theorem 2}).\\
$(iii)$ Using a probabilistic approach developed by Tzvetkov and Visciglia \cite{TV}, Y. Deng \cite{D} 
proved well-posedness result for the BO equation on the torus for almost every data with respect to a measure which is supported by $\bigcap_{\e > 0}H^{-\e}_r$ 
and for which $L^2_r$ is of measure $0$. Our result provides a deterministic framework for these solutions.

\smallskip

 One of the key ingredients of our proof of Theorem \ref{Theorem 1} 
 are explicit formulas for the frequencies of the Benjamin-Ono equation, defined by
 \eqref{frequencies in Birkhoff 0} below, which describe the time evolution of solutions of \eqref{BO}
 when expressed in Birkhoff coordinates.
 They are not only used to prove the global well-posedness results for \eqref{BO},
 but at the same time allow to obtain the following qualitative properties of solutions of \eqref{BO}.
  
\begin{theorem}\label{Theorem 3}
For any $v_0 \in H^{-s}_{r, c}$ with $0 < s < 1/2$ and $c \in \R$,
the solution $ \mathcal S(t, v_0)$
has the following properties:\\
$(i)$ The orbit $\{ \mathcal S(t, v_0) \, : \, t \in \R \}$
is relatively compact in $H^{-s}_{r, c}$.\\
$(ii)$ The solution $t \mapsto \mathcal S(t, v_0)$
is almost periodic in $H^{-s}_{r, c}$.
\end{theorem}
\begin{remark}
Theorem \ref{Theorem 3} continues to hold for any initial data in $H^s_{r, c}$ with $s > 0$ arbitrary.
See Corollary \ref{results for s positive} in Appendix \ref{restrictions}. 
For $s=0,$ results corresponding to the ones of Theorem \ref{Theorem 3} have been obtained in \cite{GK}.
\end{remark}

In \cite{AT}, Amick and Toland characterized the traveling wave solutions
of \eqref{BO}, originally found by Benjamin \cite{Benj} . It was shown in \cite[Appendix B]{GK} that they coincide
with the so called one gap solutions, described explicitly in \cite{GK}.
Note that one gap potentials are rational solutions of \eqref{BO} and
evolve in $\bigcap_{n \ge 1}H^{n}_{r,0}$.
In \cite[Section 5.1]{AN} Angulo Pava and Natali proved that every travelling
wave solution of \eqref{BO} is orbitally stable in $H^{1/2}_r$. Our newly developed methods
allow to complement their result as follows:
\begin{theorem}\label{Theorem 4}
Every traveling wave solution of the BO equation
is orbitally stable in $H^{-s}_r$ for any $0 \le s < 1/2$. 
\end{theorem}
\begin{remark}
Theorem \ref{Theorem 4} continues to hold on $H^s_{r}$ for any $s > 0$.
See Corollary \ref{results for s positive} in Appendix \ref{restrictions}. 
\end{remark}

\smallskip

\noindent
{\bf Method of proof.} Let us explain our method for studying low regularity solutions of integrable PDEs / $\Psi$DEs such as the Benjamin-Ono equation, 
in an abstract, informal way.  
Consider an integrable evolution equation (E) of the form $\partial_t u = X_{\mathcal H} (u)$ where $X_{\mathcal H} (u)$ denotes the Hamiltonian
vector field, corresponding to the Hamiltonian $\mathcal H$. In a first step we disregard the equation (E) and 
choose instead a family of Poisson commuting Hamiltonians $\mathcal H_\lambda$, parametrized by $\lambda \in \Lambda$, 
with the property that the Hamiltonian $\mathcal H$ is in the Poisson algebra, generated by the family $(\mathcal H_\lambda)_{\lambda \in \Lambda}$, 
i.e.,  $\{\mathcal H, \mathcal H_\lambda  \} = 0$ for any $\lambda \in \Lambda$. To study low regularity solutions of (E), 
the choice of $\mathcal H_\lambda$, $\lambda \in \Lambda$, has to be made judiciously. Typically, the so called hierarchies, 
often associated with integrable PDEs/$\Psi$DEs are not well suited families. Our strategy is to choose such a family with the help of a Lax pair formulation of (E), 
$\partial_t L = [B, L]$ where $L\equiv L_u$ and $B\equiv B_u$ are  typically differential or pseudo-differential operators acting on Hilbert spaces of functions, 
with symbols depending on $u$, and where $[B, L]$ denotes the commutator of $B$ and $L$.
At least formally, the spectrum of the operator $L$ is conserved by the flow of (E). The goal
is to find a Lax pair $(L, B)$ for (E) with the property that the operator $L$ is well defined for
$u$ of low regularity and then choose functions $\mathcal H_\lambda$, encoding the spectrum of $L$, 
such as the (appropriately regularized) determinant of $L - \lambda$ or a perturbation determinant. We refer to such a function as a generating function.
The key properties of $\mathcal H_\lambda$  to be established are the following ones: 
(i) the flows of the Hamiltonian vector fields $X_{\mathcal H_\lambda}$ are well defined for $u$
of low regularity and can be integrated globally in time; (ii) for $u$ sufficiently regular, $\mathcal H$ can be expressed in terms of the generating function;
(iii) the generating function can be used to construct Birkhoff coordinates
so that the Hamiltonian vector field $X_{\mathcal H}$, when expressed in these coordinates,
extends to spaces of $u$ of low regularity.\\
In the case of the Benjamin-Ono equation, this method is implemented as follows. 
In a first important step we prove that the operator $L_u$ (cf. \eqref{Lax operator}) of the Lax pair
for the Benjamin-Ono equation, found by Nakamura \cite{Nak}, has the property that it is well defined for $u$
in the Sobolev spaces $H^{-s}_{r,0}$, $0 < s < 1/2$. See the paragraph {\em Ideas of the proof of Theorem 5} in Section \ref{Birkhoff map}
for more details. By \eqref{definition generation fucntion} in Section \ref{extension of Phi, part 1},
our choice of the generating function is $\mathcal H_\lambda(u) = \langle (L_u + \lambda)^{-1} 1 | 1 \rangle$
and the Hamiltonian $\mathcal H$ of the Benjamin-Ono equation, when expressed in Birkhoff coordinates,
is given by \eqref{Hamiltonian in Birkhoff}.
The novel conservation laws
of the Benjamin-Ono equation of Theorem \ref{Theorem 1}, $I_{-s} : H^{-s}_r \to \R_{\ge 0}$,
together with the results on the Lax operator $L_u$ for $u$ in $H^{-s}_{r,0}$
are the key ingredients to construct Birkhoff coordinates on $H^{-s}_{r,0}$ for any $0 < s < 1/2$. 
When expressed in these coordinates,  equation \eqref{BO} can be solved by quadrature.

\smallskip

\noindent
{\bf Related work.} 
Results on global well-posedness of the type
stated in Theorem \ref{Theorem 1} have been obtained for other integrable PDEs such as the KdV, the KdV2,
the mKdV, and the defocusing NLS equations. A detailed analysis of the frequencies of these equations
allowed to prove in addition to the well-posedness results qualitative properties of solutions of these equations,
among them properties corresponding to the ones stated in Theorem \ref{Theorem 3} 
 -- see e.g. \cite{KT1},\cite{KT2}, \cite{KM1}, \cite{KM2}. 
 Very recently, sharp global well-posedness results for the cubic NLS, the mKdV equation,
 the KdV equation, and the fifth-order KdV equation  on the real line
 were obtained in \cite{HKV}, \cite{KV}, and, respectively, \cite{BKV}. They are based on
novel integrals constructed in \cite{KVZ}  (cf. also \cite{KT}). 
By the same method, Killip and Visan provide in \cite{KV} alternative proofs 
of the  global well-posedness results for the KdV equation on the torus obtained in \cite{KT1}.
However, to the best of our knowledge, their method does not allow to deduce qualitative properties 
of solutions of the KdV equation on $\T$ such as almost periodicity nor to obtain coordinates
 which can be used to study perturbations of the KdV equation by KAM type methods.
 
\smallskip

\noindent
{\bf Subsequent work.}  
 One of the main novel features of the Benjamin--Ono equation, when compared from the point of view of integrable PDEs with
 the KdV equation or the cubic NLS equation, is that the Lax operator $L_u$ (cf. \eqref{Lax operator}), appearing in the Lax pair formulation of \eqref{BO},
  is {\em nonlocal}. One of the consequences of $L_u$ being nonlocal
 is that the study of the regularity of the Birkhoff map and of its restrictions to the scale of Sobolev spaces $H^s_{r,0}$, $s \ge 0,$
 is quite involved. 
 Further results on the Birkhoff map of the Benjamin-Ono equation in this direction will be reported on in subsequent work.
 
 \smallskip

\noindent
{\bf Organisation.} In Section \ref{Birkhoff map}, we state our results on the extension of  the Birkhoff map  $\Phi$ (cf. Theorem \ref{extension Phi})
 and discuss first applications. All these results are proved 
in Section \ref{extension of Phi, part 1}  and Section \ref{extension of Phi, part 2}.
In Section \ref{mathcal S_B},
we study the solution map $\mathcal S_B$
corresponding to the system of equations,
obtained when expressing \eqref{BO} in Birkhoff coordinates. 
These results are then used
to study the solution map $\mathcal S$ of \eqref{BO}.
In the same section
we also introduce the solution map $\mathcal S_c$ (cf. \eqref{solution map for BO with c}), 
defined in terms of the solution map of the equation \eqref{BO} in the affine
space $H^s_{r,c}$, $c \in \R$, and study the
solution map $\mathcal S_{c, B}$, obtained by
expressing $\mathcal S_c$ in Birkhoff coordinates.
With all these preparations done, we prove Theorem \ref{Theorem 1}, Theorem \ref{Theorem 3}, and Theorem \ref{Theorem 4},
 in Section \ref{Proofs of main results}. The proof of Theorem \ref{Theorem 2}  is presented in Section \ref{Proof of Theorem 2}.
Finally, in Appendix \ref{restrictions}
we study the restriction of the Birkhoff map to the Sobolev spaces $H^s_{r,0}$ with $s > 0$
and discuss applications to the Benjamin-Ono equation, while in Appendix \ref{illposedness for s < -1/2} we discuss
results on ill-posedness of the Benjamin--Ono equation in $H^{-s}_r$ with $s > \frac 12$. 

\smallskip

\noindent
{\bf Notation.}
By and large, we will use the notation established in
\cite{GK}. In particular, 
the $H^s-$norm of an element $v$ in the 
Sobolev space $H^s \equiv H^s(\T, \C)$, $s \in \R$,
will be denoted by $\|v\|_s$. It is defined by 
\begin{equation}\label{Hs norm}
\|v\|_s = 
\big( \sum_{n \in \Z} \langle n \rangle^{2s}
|\widehat v(n)|^2 \big)^{1/2}\, , \quad
 \langle n \rangle = \max\{1, |n|\}\, .
\end{equation}
For $\|v\|_0$, we usually write $\|v\|$. 
By $\langle \cdot | \cdot \rangle$, we will also
denote the extension of the $L^2-$inner product,
introduced in \eqref{L2 inner product}, 
to $ H^{-s}\times H^s$, $s \in \R$, by duality.
By $H_+$ we denote the Hardy space, consisting
of elements $f \in L^2(\T, \C) \equiv H^0$ with
the property that 
$ \widehat f(n) = 0$ for any $n < 0$.
More generally, for any $s \in \R$,  $H^{s}_+$
denotes the subspace of $H^s,$
consisting of elements $f \in H^s$ with the property
that $ \widehat f(n) = 0$ for any $n < 0$.

\smallskip

\noindent
{\bf Previous versions.} A first version of this paper appeared on arXiv in September 2019 and a second one
with additional results in December 2019 -- see \cite{GKT}. In the current version, Section \ref{Proof of Theorem 2} (proof of ill-posedness of \eqref{BO} in $H^{-1/2}_r)$,
Appendix \ref{restrictions} (restriction of the Birkhoff map to $H^s_{r.0}$ and applications), and 
Appendix \ref{illposedness for s < -1/2} (ill-posedness of \eqref{BO} in $H^{-s}_{r}$ for $s> 1/2$) have been added and the introduction has been 
extended. To reflect better the content of the current version,  the title of the paper has been changed.


\section{The Birkhoff map $\Phi$}\label{Birkhoff map}

In this section we present our results on Birkhoff coordinates
which will be a key ingredient of the proofs of
Theorem \ref{Theorem 1} -- Theorem \ref{Theorem 4}. 
We begin by reviewing the results
on Birkhoff coordinates proved in \cite{GK}.
Recall that on appropriate Sobolev spaces,  
\eqref{BO} can be written in Hamiltonian form  
$$
\partial_t u = \partial_x (\nabla \mathcal H (u))\,, \qquad  \mathcal H (u):=  \frac{1}{2\pi}\int_0^{2 \pi} 
\big( \frac{1}{2} 
(|\partial_x|^{1/2} u)^2 - \frac{1}{3} u^3 
\big) dx
$$
where $|\partial_x|^{1/2}$ is the square root of the Fourier multiplier operator $|\partial_x|$ 
given by
$$
|\partial_x| f(x) = \sum_{n \in \Z} |n| \widehat f(n) e^{inx}\,.
$$
Note that the $L^2-$gradient $\nabla \mathcal H$ of $\mathcal H$ can be computed to be $|\partial_x| u - u^2$ and that $\partial_x \nabla \mathcal H$ is the
Hamiltonian vector field
corresponding to the Gardner bracket, defined for any two functionals $F, G : H^0_{r} \to \R$ with sufficiently regular $L^2-$gradients by
$$
 \{F, G \} := \frac{1}{2 \pi} \int_0^{2\pi} (\partial_x \nabla F) \nabla G dx\ .
$$
In \cite{GK}, it is shown that \eqref{BO} admits global Birkhoff coordinates and hence is an integrable $\Psi$DE in the strongest possible sense. 
To state this result in more detail, we first introduce some notation.
For any subset $J \subset \N_0 :=\mathbb Z_{\ge 0}$ 
and any $s \in \mathbb R$, 
$h^s(J) \equiv h^s (J, \mathbb C)$ denotes the weighted $\ell^2-$sequence space
$$
h^s(J) = \{ (z_n)_{n \in J} \subset \mathbb C \, : \, \| (z_n)_{ n \in J} \|_s < \infty  \}
$$
where
$$
\| (z_n)_{ n \in J} \|_s : = \big( \sum_{n \in J} \langle n \rangle^{2s} |z_n|^2 \big)^{1/2} \ , \quad  
\langle n \rangle := \text{ max} \{ 1, |n| \} \, .
$$
By $h^s(J, \R)$, we denote the real subspace of 
$h^s(J, \C)$, consisting of real sequences $(z_n)_{n \in J}$.
In case where $J = \mathbb N := \{ n \in \mathbb Z \, : \, n \ge 1 \}$ we write
$h^s_+$ instead of $h^s(\mathbb N)$.
If $s=0,$ we also write $\ell^2$ instead of $h^0$ and
$\ell^2_+$ instead of $h^0_+$.
In the sequel, we view $h^s_+$ as the $\R-$Hilbert space $h^s(\N, \R) \oplus h^s(\N, \R)$
by identifying a sequence $(z_n)_{n \in \N} \in h^s_+$ with the pair of sequences
$\big( ({\rm Re} \, z_n)_{n \in \N}, ({\rm Im} \, z_n)_{n \in \N} \big)$ in $h^s(\N, \R) \oplus h^s(\N, \R)$.
We recall that  $L^2_r = H^{0}_r$ and $L^2_{r,0} = H^0_{r,0}$.
The following result was proved in \cite{GK}:
\begin{theorem}\label{main result}(\cite[Theorem 1]{GK}) 
There exists a homeomorphism 
$$
\Phi : L^2_{r,0}\to h^{1/2}_+ \,, \, u \mapsto (\zeta_n(u))_{n \ge 1}
$$
so that the following holds:\\
(B1) For any $n \ge 1$, $\zeta_n:L^2_{r,0}\to \C $ is real analytic.\\
(B2) The Poisson brackets between the coordinate functions $\zeta_n$ are well-defined and for any $n, k \ge 1,$
\begin{equation}\label{standard bracket}
\{\zeta_n , \overline{\zeta_k} \} = - i \delta_{nk}\,, \qquad  \{\zeta_n , \zeta_k \} = 0\,.
\end{equation} 
It implies that the functionals $|\zeta_n|^2$, $n \ge 1$,
pairwise Poisson commute,
$$
\{|\zeta_n|^2 , |\zeta_k|^2 \} = 0\, , 
\quad \forall n,k \ge 1 \, .
$$
(B3) On its domain of definition, $\mathcal H \circ \Phi^{-1}$ is a (real analytic) function, which only depends on the actions $|\zeta_n|^2,$ $n \ge 1$. 
As a consequence, for any $n \ge 1$,
$|\zeta_n|^2$ is an integral of $\mathcal H  \circ \Phi^{-1}$,
$\{ \mathcal H \circ \Phi^{-1}, |\zeta_n|^2 \} = 0$.
\\
The coordinates $\zeta_n$, $n \ge 1$, are referred to as complex Birkhoff coordinates and the functionals 
$|\zeta_n|^2$, $n \ge 1$,  as action variables.
\end{theorem}
\begin{remark}\label{RemarkThm1}
$(i)$ When restricted to submanifolds of finite gap potentials (cf. \cite[Definition 2.2 ]{GK}),
the map $\Phi$ is a canonical, real analytic diffeomorphism
onto corresponding Euclidean spaces -- see 
\cite[Theorem 3]{GK} for details.\\
$(ii)$  For any bounded subset $B$ of $L^2_{r,0}$, 
the image  $\Phi(B)$ by $\Phi$ is bounded in $h^{1/2}_+$.
This is a direct consequence of the trace formula, saying that for any $u \in L^2_{r,0}$ (cf. \cite[Proposition 3.1]{GK}), 
\begin{equation}\label{trace formula}
\|u \|^2 = 2 \sum_{n =1}^\infty n |\zeta_n|^2 \ .
\end{equation}  
\end{remark}

Theorem \ref{main result} together with Remark \ref{RemarkThm1}(i) can be used to solve 
the initial value problem of
\eqref{BO} in $L^2_{r,0}$.
Indeed, by approximating a given initial data in $L^2_{r,0}$ by finite gap potentials (cf. \cite[Definition 2.2 ]{GK}),
one concludes from \cite[Theorem 3]{GK} and Theorem \ref{main result} that equation $\eqref{BO}$, 
when expressed in the Birkhoff coordinates $\zeta = (\zeta_n)_{n \ge 1}$, reads
\begin{equation}\label{BO in Birkhoff}
\partial_t \zeta_n = 
\{\mathcal H \circ \Phi^{-1}, \zeta_n  \} =
i \omega_n \zeta_n \, , 
\quad \forall n \ge 1\, ,
\end{equation}
where 
$\omega_n$, $n \ge 1$, are the BO frequencies,
\begin{equation}\label{frequencies in Birkhoff 0}
\omega_n = 
\partial_{|\zeta_n|^2} \mathcal H \circ \Phi^{-1} \, .
\end{equation}
Since the frequencies only depend on the actions $|\zeta_k|^2$,
$k \ge 1$, they are conserved and hence \eqref{BO in Birkhoff}
can be solved by quadrature,
\begin{equation}\label{BOsolution}
\zeta_n(t)=\zeta_n(0)\, 
e^{i\omega_n(\zeta(0)) t}\ ,\quad t\in \R , \quad n \ge 1\, .
\end{equation}
By \cite[Proposition 8.1]{GK}),  
$\mathcal H_B := \mathcal H \circ \Phi^{-1} $
can be computed as
\begin{equation}\label{Hamiltonian in Birkhoff}
\mathcal H_B (\zeta) := \sum_{k=1}^\infty k^2 |\zeta_k|^2 -
\sum_{k=1}^\infty (\sum_{p = k }^\infty |\zeta_p|^2 )^2\, ,
\end{equation}
implying that the frequencies, defined by 
\eqref{frequencies in Birkhoff 0}, are given by
\begin{equation}\label{frequencies in Birkhoff}
\omega_n(\zeta) = n^2 - 
2 \sum_{k=1}^{\infty} \min(n,k) |\zeta_k|^2 \, ,
\quad \forall n \ge 1\, .
\end{equation}
Remarkably, for any $n \ge 1$, $\omega_n$ depends 
{\em linearly} on the actions 
$|\zeta_k|^2$, $k \ge 1$. Furthermore, while the Hamiltonian
$\mathcal H_B$ is defined on $h^{1}_+$, 
the frequencies $\omega_n$, $n \ge 1$, given by
\eqref{frequencies in Birkhoff} for $\zeta \in h^1_+,$
extend to bounded functionals on $\ell^2_+$,
\begin{equation}\label{frequency map}
\omega_n: \ell^2_+ \to \R,\, 
\zeta = (\zeta_k)_{k \ge 1} \mapsto \omega_n(\zeta)\, .
\end{equation}
We will prove 
that the restriction $\mathcal S_0$ 
of the solution map of \eqref{BO}
to $L^2_{r,0}$, when expressed in
Birkhoff coordinates,
$$
\mathcal S_B : h^{1/2}_+ \to C(\R, h^{1/2}_+) \, , \
\zeta(0) \mapsto 
(\zeta_n(0)\, e^{i\omega_n(\zeta(0)) t})_{n \ge 1}
$$
is continuous -- see Proposition \ref{S in Birkhoff continuous} in Section \ref{mathcal S_B}.
By Theorem \ref{main result},
$\Phi: L^2_{r,0} \to h^{1/2}_+$ and
its inverse $\Phi^{-1}: h^{1/2}_+ \to L^2_{r,0}$
are continuous. Since
$$
\mathcal S_0 = \Phi^{-1} \mathcal S_B \Phi\, : \,  
L^2_{r,0} \to C(\R, L^2_{r,0}) \, , \, 
u(0) \mapsto  \Phi^{-1}\mathcal S_B(t, \Phi(u(0)))
$$
it follows that 
$\mathcal S_0 : L^2_{r,0} \to C(\R, L^2_{r,0})$
is continuous as well. We remark that for any
$u(0) \in L^2_{r,0}$, the solution
$t \mapsto \mathcal S(t, u(0))$ 
can be approximated
in $L^2_{r,0}$ by classical solutions of equation \eqref{BO} (cf. Remark \ref{RemarkThm1}$(i)$)
and thus coincides with the solution,
obtained by Molinet in \cite{Mol} (cf. also \cite{MP}). 

Starting point of the proof of Theorem \ref{Theorem 1}
is formula \eqref{flow map BO} in Subsection \ref{solution map Sc}. We will show that it
extends to the Sobolev spaces $H^{-s}_{r,0}$ for
any $0 < s < 1/2$.
A key ingredient to prove Theorem \ref{Theorem 1} is
therefore the following result on the extension 
of the Birkhoff map  $\Phi$ to 
$H^{-s}_{r,0}$ for any $0 < s < 1/2$:
\begin{theorem}\label{extension Phi}
{\sc (Extension of $\Phi$.)}
For any $0 < s < 1/2,$ the map $\Phi$ of 
Theorem \ref{main result} admits an extension, 
also denoted by $\Phi$,
$$\Phi :H^{-s}_{r,0} \rightarrow h^{1/2-s}_+, \,\,  u \mapsto 
\Phi (u):=(\zeta_n(u))_{n\ge 1}\, ,$$
so that the following holds: \\
$(i)$ $\Phi$ is a homeomorphism. \\
$(ii)$ There exists an increasing function 
$F_s : \R_{>0} \to \R_{>0}$ so that 
$$
\|u\|_{-s} \le F_s( \| \Phi(u) \|_{1/2 - s}) 
\, , \qquad \forall u \in H^{-s}_{r,0} \, .
$$
$(iii)$ $\Phi$ and its inverse map
bounded subsets to bounded subsets.
\end{theorem}
\begin{remark}\label{weakPhi}
$(i)$ The Birkhoff map does not continuously extend to $H^{-1/2}_{r,0}$ -- see Corollary \ref{no extension for s=-1/2}
at the end of Section \ref{Proof of Theorem 2}.\\
$(ii)$ Results, developed in the course of the proof of Theorem \ref{extension Phi} allow to study the restriction of the Birkhoff map
to $H^s_r$ for any $s > 0$. See  Proposition \ref{positive} in Appendix \ref{restrictions} for details. \\
$(iii)$ Items $(i)$ and $(iii)$, combined with the Rellich compactness theorem, imply that for $0 \leq s < \frac 12$,
the map $\Phi :H^{-s}_{r,0} \rightarrow h^{1/2-s}_+$ and its inverse $\Phi^{-1}:  h^{1/2-s}_+ \rightarrow  H^{-s}_{r,0}$
 are weakly sequentially continuous on $H^{-s}_{r,0}$.
\end{remark}
\begin{remark}\label{def integral I_s}
The above a priori bound for $\| u \|_{-s}$
can be extended to the space $H^{-s}_r$ as follows
$$
 \| v \|_{-s} 
\le F_s(\| \Phi(v - [v]) \|_{1/2 - s}) + |[v]| \, ,
\quad [v] = \langle v | 1 \rangle \, , \qquad
\forall v \in H^{-s}_r \, .
$$
For any $0 < s < 1/2$, the integral $ I_{-s}$ in 
Theorem \ref{Theorem 1}$(iv)$ is defined as
$$
I_{-s}(v) := F_s(\| \Phi(v - [v]) \|_{1/2 - s}) + |[v]| \, .
$$
\end{remark}

\medskip

\noindent
{\em Ideas of the proof of Theorem \ref{extension Phi}.} 
At the heart of the proof of Theorem 1 in \cite{GK} is the Lax operator $L_u$,
appearing in the Lax pair formulation in \cite{Nak} (cf. also \cite{BK}, \cite{CW}, \cite{FA})
$$
\partial_t L_u = [B_u, L_u]
$$
of \eqref{BO} -- see \cite[Appendix A]{GK} for a review.
For any given $u \in L^2_{r}$, the operator $L_u$
is the first order operator acting on the Hardy space $H_+$,
\begin{equation}\label{Lax operator}
L_u := -i \partial_x  - T_u\, , \qquad 
T_u (\cdot) := \Pi(u \, \cdot)
\end{equation}
where $\Pi$ is the orthogonal projector of $L^2$ onto $H_+$ and  $T_u$ is the Toeplitz operator with symbol $u$, 
$$
H_+ := 
\{ f \in L^2 \ : \ \widehat f(n) = 0 \, \, \, \forall n < 0  \} \, .
$$
The operator $L_u$ is self-adjoint with domain 
$H^1_+:= H^1 \cap H_+$, bounded from below, and has a compact resolvent. Its spectrum consists of real eigenvalues which
are bounded from below.
When listed in increasing order 
they form a sequence, satisfying
$$
\lambda_0 \le \lambda_1 \le \cdots \, ,   \qquad
\lim_{n \to \infty}\lambda_n = \infty \, .
$$
For our purposes, the most important properties of 
the spectrum of $L_u$ are that 
the eigenvalues are conserved along the flow of \eqref{BO}
and that they are all simple. More precisely, one has
\begin{equation}\label{def nth gap}
\gamma_n := \lambda_n - \lambda_{n-1} -1 \ge 0\, , \quad
\forall n \ge 1\, .
\end{equation}
The nonnegative number $\gamma_n$ is referred to  
as the $n$th gap
of the spectrum ${\rm{spec}}(L_u)$ of $L_u$.
-- see \cite[Appendix C]{GK} for an explanation of this terminology. For any $n \ge 1$, the complex 
Birkhoff coordinate $\zeta_n$
of Theorem \ref{main result} is related to $\gamma_n$ by
$|\zeta_n|^2 = \gamma_n$ whereas its phase is defined
in terms of an  appropriately normalized eigenfunction $f_n$
of $L_u$, corresponding to the eigenvalue $\lambda_n$.\\
A key step for the proof of Theorem \ref{extension Phi} 
is to show that
for any $u \in H^{-s}_{r}$ with $0 < s < 1/2$, 
the Lax operator $L_u$
can be defined as a self-adjoint operator with domain included in
$H^{1-s}_+$ and that its spectrum has properties
similar to the ones described above in the case where
$u \in L^2_r$. In particular, the inequality \eqref{def nth gap} continues to hold.
Since  the proof of Theorem \ref{extension Phi} requires several steps, it
is split up into  two parts, corresponding to
Section \ref{extension of Phi, part 1} and 
Section \ref{extension of Phi, part 2}.

\smallskip

A straightforward application of Theorem \ref{extension Phi}
is the following result on isospectral potentials.
To state it, we need to introduce some additional notation. For any $\zeta \in h^{1/2-s}_+$, define
\begin{equation}\label{def tor}
{\rm{Tor}}(\zeta) := \{ z \in h^{1/2-s}_+ \, : \,
|z_n| = |\zeta_n| \, \,\,\forall n \ge 1 \}.
\end{equation}
Note that ${\rm{Tor}}(\zeta)$ is an infinite product of (possibly degenerate) circles
and a compact
subset of $h^{1/2-s}_+$. Furthermore, for any 
$u \in H^{-s}_{r,0}$, let
$${\rm{Iso}}(u) := \{ v \in H^{-s}_{r,0} \, : \,
 {\rm{spec}}(L_v) = {\rm{spec}}(L_u) \} \, .
$$ 
where as above, ${\rm{spec}}(L_u)$ denotes the spectrum
of the Lax operator $L_u := -i \partial_x  - T_u$.
The spectrum of $L_u$ continues to be
characterized in terms of its gaps $\gamma_n$, $n \ge 1$, 
(cf. \eqref{def nth gap}) and
the extended Birkhoff coordinates continue to satisfy
$| \zeta_n |^2= \gamma_n$, $n \ge 1$.
An immediate consequence of Theorem \ref{extension Phi}
then is that \cite[Corollary 8.1]{GK} extends as follows:
\begin{corollary}\label{isospectral set}
For any $u \in H^{-s}_{r,0}$ with $0 < s < 1/2$, 
$$
\Phi({\rm{Iso}}(u)) = {\rm{Tor}}(\Phi(u)) \, .
$$
Hence by the continuity of $\Phi^{-1}$,
${\rm{Iso}}(u)$ is a compact, connected subset
of $H^{-s}_{r,0}$.
\end{corollary}


\section[extension of Birkhoff map. Part 1]{Extension of $\Phi$. Part 1}\label{extension of Phi, part 1}

In this section we prove the first part of 
Theorem \ref{extension Phi}, which we state
as a separate result:
\begin{proposition}\label{extension Phi, part 1}
{\sc (Extension of $\Phi$. Part 1)}
For any $0 < s < 1/2$, the following holds:\\
$(i)$ For any $n \ge 1$, the formula in \cite[(4.1)]{GK} of
the Birkhoff coordinate
$\zeta_n : L^2_{r,0} \to \C$ extends to $H^{-s}_{r,0}$ 
and for any $u \in H^{-s}_{r,0}$,
$(\zeta_n(u))_{n \ge 1}$ is in $h^{1/2-s}_+$.
The extension of the map $\Phi$ of Theorem \ref{main result}, 
also denoted by $\Phi$,
$$\Phi :H^{-s}_{r,0} \rightarrow h^{1/2-s}_+, \,\,  u \mapsto 
\Phi (u):=(\zeta_n(u))_{n\ge 1}\, ,$$
 maps bounded subsets of $H^{-s}_{r,0}$ 
to bounded subsets of $h^{1/2-s}_+$.\\
$(ii)$ $\Phi$ is sequentially weakly continuous
and one-to-one.
\end{proposition}

 First we need to establish some auxiliary results
 related to the Lax operator $L_u$.

 \begin{lemma}\label{estimate for T_u f}
 Let $u \in H^{-s}_{r,0}$ with $0 \le s < 1/2$.
 Then for any $f, g \in H^{1/2}_+$, 
 the following estimates hold:\\
 $(i)$ There exists a constant $C_{1,s} >0$ only depending
 on $s$, so that
 \begin{equation}\label{para}
 \| f g \|_s \le 
  C_{1,s}^2 \| f\|_\sigma \| g \|_\sigma\, , \qquad
 \sigma:= (1/2 +s)/2 \,.
 \end{equation}
$(ii)$ The expression $\langle  u | f \overline f \rangle$
is well defined and
satisfies the estimate
 \begin{equation}\label{estimate Toeplitz}
 | \langle  u | f \overline f \rangle |
 \le \frac{1}{2} \| f \|_{1/2}^2 + 
 \eta_s(\| u \|_{-s}) \|f \|^2
 \end{equation}
 where
 \begin{equation}\label{def eta_s}
 \eta_s(\| u \|_{-s}) := \|u\|_{-s}
 \big( 2 ( 1 + \|u\|_{-s}))^{\alpha} C_{2,s}^2 \, ,
 \quad \alpha:= \frac{1+2s}{1-2s} \, 
 \end{equation}
 and $C_{2,s} >0$ is a constant, only depending on $s$.
 \end{lemma}
 \begin{proof}
 $(i)$ Estimate \eqref{para} is obtained from standard
 estimates of paramultiplication
 (cf. e.g. \cite[Exercise II.A.5]{AG}, 
 \cite[Theorem 2.82, Theorem 2.85]{BCD}).
 $(ii)$ By item $(i)$,
 $\langle  u | f \overline f \rangle$
 is well defined by duality and satisfies
 $$
 | \langle  u | f \overline f \rangle |
 \le \| u \|_{-s} \|f \overline f \|_s 
 \le  \| u \|_{-s} C_{1,s}^2 \| f\|_\sigma^2 \, .
 $$
 In order to estimate $\| f\|_\sigma^2$, note that 
 by interpolation one has  
$ \| f\|_\sigma \le \| f \|_{1/2}^{1/2 +s}
  \| f \|^{1/2 -s}$ and hence 
 \begin{equation}\label{interpolation}
  C_{1,s} \| f\|_\sigma \le \| f \|_{1/2}^{1/2 +s}
  \big( C_{2,s} \| f \| \big)^{1/2 -s}
 \end{equation}
 for some constant $C_{2,s}> 0$.
 Young's inequality then yields
 for any $\e > 0$
 \begin{equation}\label{Young}
 \big( C_{1,s} \| f\|_\sigma \big)^2 \le
 \e \| f \|_{1/2}^2 + 
 \e^{-\alpha} \big( C_{2,s} \| f \| \big)^2\, ,
 \qquad \alpha = \frac{1+2s}{1-2s} \, .
 \end{equation}
 Estimate \eqref{estimate Toeplitz} then
 follows from \eqref{Young} by
 choosing $\e = \big( 2 ( 1 + \|u\|_{-s}))^{-1}$. 
 \end{proof}
 Note that estimate \eqref{estimate Toeplitz} implies
 that the sesquilinear form $\langle T_u f | g \rangle $
 on $H^{1/2}_+$, 
 obtained from the Toeplitz operator $T_u f := \Pi (u f)$
 with symbol $u \in L^2_{r,0}$,
 can be defined  for any $u \in H^{-s}_{r,0}$ 
 with $0 \le s < 1/2$ by 
setting 
$\langle T_u f | g \rangle := 
\langle u | g \overline f  \rangle$ and that it is bounded.
 For any $u \in H^{-s}_{r,0},$ we then define the 
 sesquilinear form $Q_u^+$ on $H^{1/2}_+$ as follows
 \begin{equation}\label{Q_u^+}
 Q_u^+(f, g) : = \langle -i\partial_x f | g \rangle
 - \langle T_u f | g \rangle +
 \big( 1 + \eta_s(\|u\|_{-s}) \big) \langle f | g \rangle
 \end{equation}
 where $\eta_s(\|u\|_{-s})$ is given by \eqref{def eta_s}. 
 The following lemma says that the quadratic form 
 $Q^+_u(f, f)$ is equivalent to $\|f\|_{1/2}^2$.
 More precisely, the following holds.
 \begin{lemma}\label{comparison Q_u^+}
 For any $ u \in H^{-s}_{r,0}$ with $0 \le s < 1/2$,
 $Q^+_u$ is a positive, sesquilinear form,
 satisfying 
 $$
 \frac{1}{2} \| f \|_{1/2}^2 \le  Q^+_u(f, f) \le 
 \big( 3 + 2 \eta_s(\|u\|_{-s}) \big) \| f \|_{1/2}^2 \, ,
 \quad \forall f \in H^{1/2}_+\, .
 $$
 \end{lemma}
 \begin{proof}
 (i) Using that $u$ is real valued, one verifies that
 $Q_u^+$ is sesquilinear.
 The claimed estimates are obtained from 
 \eqref{estimate Toeplitz} as follows:
 since $\langle n \rangle \le 1 + |n|$
 one has
 $
 \| f \|_{1/2}^2 \le 
 \langle -i \partial_x f | f \rangle  + \| f \|^2$,
 and hence by \eqref{estimate Toeplitz},
 $$
 | \langle T_u f | f \rangle | \le 
 \frac{1}{2} \langle -i \partial_x f | f \rangle
 + \big(\frac{1}{2} + \eta_s(\|u\|_{-s}) \big)\| f \|^2 \, .
 $$
 By the definition \eqref{Q_u^+}, the claimed estimates 
 then follow.
 In particular, the lower bound for $Q_u^+(f, f)$ shows that
 $Q_u^+$ is positive.
 \end{proof}
 Denote by 
 $\langle f | g \rangle_{1/2} \equiv \langle f | g \rangle_{H^{1/2}_+}$
 the inner product, corresponding to the norm 
 $\| f \|_{1/2}$. It is given by
 $$
 \langle f |  g \rangle_{1/2} =
 \sum_{n \ge 0} \langle n \rangle 
 \widehat f(n) \overline{\widehat g (n)}
 \, , \quad \forall f, g \in H^{1/2}_+ \, .
 $$
 Furthermore, denote by 
 $D : H^t_+ \to H^{t-1}_+$ and
 $\langle D \rangle : H^t_+ \to H^{t-1}_+$, $t\in \R$, 
 the Fourier multipliers, defined for $f \in H^t_+$
 with Fourier series
 $f = \sum_{n= 0}^\infty \widehat f(n) e^{inx}$ by 
 $$
  D f := -i\partial_x f = \sum_{n = 0}^\infty 
 n  \widehat f(n) e^{inx} \, , \qquad
 \langle D \rangle f := \sum_{n = 0}^\infty 
 \langle n \rangle \widehat f(n) e^{inx} \, .
 $$
 \begin{lemma}\label{Lax Milgram}
 For any $ u \in H^{-s}_{r,0}$ with $0 \le s < 1/2$,
 there exists a bounded linear isomorphism 
 $A_u : H^{1/2}_+ \to H^{1/2}_+$ so that
 $$
 \langle A_u f | g \rangle_{1/2} =
 Q^+_u(f, g)\, , \quad \forall f, g \in H^{1/2}_+\, .
 $$
 The operator $A_u$ has the following properties:\\
 $(i)$ $A_u$ and its inverse $A_u^{-1}$ are symmetric,
 i.e., for any $f, g \in H^{1/2}_+,$
 $$
 \langle A_u f | g \rangle_{1/2} =
 \langle  f | A_u g \rangle_{1/2}\, , \quad
 \langle A_u^{-1} f | g \rangle_{1/2} =
 \langle  f | A_u^{-1}g \rangle_{1/2} \, .
 $$
 $(ii)$ The linear isomorphism $B_u$, given by the composition
 $$
 B_u := \langle D \rangle A_u : H^{1/2}_+ \to H^{-1/2}_+
 $$
 satisfies
 $$
Q_u^+ (f, g) = \langle B_u f | g \rangle \, , 
\quad \forall f, g \in H^{1/2}_+ \, .
 $$
 The operator norm of 
 $B_u$ and the one of its inverse can be bounded uniformly
 on bounded subsets of elements $u$ in $H^{-s}_{r,0}$.
 \end{lemma}
 \begin{proof}
 By Lemma \ref{comparison Q_u^+}, the 
 sesquilinear form
 $Q_u^+$ is an inner product on $H^{1/2}_+$, equivalent to
 the inner product $\langle \cdot | \cdot \rangle_{1/2}$.
 Hence by the theorem of Fr\'echet-Riesz, for any 
 $g \in H^{1/2}_+$, there exists a unique element in 
 $H^{1/2}_+$, which we denote by $A_ug$, so that
 $$
 \langle A_u g | f \rangle_{1/2} =
 Q^+_u(g, f)\, , \quad  \forall f \in H^{1/2}_+\, .
 $$Invitation for special issue in honor of Tony Bloch in Journal  of Geometric Mechanics
 Then $A_u : H^{1/2}_+ \to H^{1/2}_+$ is a linear, injective
 operator, which by Lemma \ref{comparison Q_u^+} is bounded,
 i.e., for any $f, g \in H^{1/2}_+$,
 \begin{align}
 | \langle A_u g |  f \rangle_{1/2} \ | & =
 | Q^+_u(g, f) | \le 
 Q^+_u(g, g)^{1/2} Q^+_u(f, f)^{1/2} \nonumber \\
 & \le \big( 3 + 2 \eta_s(\|u\|_{-s}) \big) 
 \|g\|_{1/2} \| f \|_{1/2} \nonumber \, ,
\end{align}
implying that 
$\| A_u g \|_{1/2} \le 
\big( 3 + 2 \eta_s(\|u\|_{-s}) \big) \| g \|_{1/2}$.

Similarly, by the theorem of Fr\'echet-Riesz, for any 
 $h \in H^{1/2}_+$, there exists a unique element in 
 $H^{1/2}_+$, which we denote by $E_uh$, so that
$$
 \langle h | f \rangle_{1/2} =
 Q^+_u(E_u h, f)\, , \quad  \forall f \in H^{1/2}_+\, .
 $$
 Then $E_u : H^{1/2}_+ \to H^{1/2}_+$ is a linear, injective
 operator, which by Lemma \ref{comparison Q_u^+} is bounded, i.e.,
 $$
 \frac{1}{2} \|E_u h\|_{1/2}^2 \le 
  Q^+_u(E_u h, E_u h) =
  \langle h | E_u h \rangle_{1/2} \le 
  \|h \|_{1/2} \| E_u h \|_{1/2}\, ,
 $$
 implying that  
 $\|E_u h\|_{1/2} \le 2 \| h\|_{1/2}$.
Note that $A_u(E_u h) = h$ and hence $E_u$ is the inverse
of $A_u$. Therefore, $A_u : H^{1/2}_+ \to H^{1/2}_+$
is a bounded linear isomorphism. Next we show item $(i)$.
For any $f, g \in H^{1/2}_+$,
$$
 \langle g | A_u f \rangle_{1/2} =
 \overline{ \langle A_u f |  g \rangle}_{1/2} =
 \overline{Q_u^+ (f, g)} =
 Q_u^+(g, f) = \langle A_u g | f \rangle_{1/2} \, .
$$
The symmetry of $A_u^{-1}$ is proved in the same way.
Towards item $(ii)$, note that for any $f, g \in H^{1/2}_+,$
$\langle f |  g \rangle_{1/2} = 
\langle \langle D \rangle f | g \rangle$
and therefore
$$
 \langle A_u g |  f \rangle_{1/2} =
  \langle \langle D \rangle A_u g |  f \rangle\, ,
$$
implying that the operator 
$B_u = \langle D \rangle A_u : H^{1/2}_+ \to H^{-1/2}_+$  
is a bounded linear isomorphism and that
$$
\langle B_u g |  f \rangle = Q_u^+(g, f) \, , \quad
\forall g,f \in H^{1/2}_+ \, . 
$$
The last statement of (ii) follows from Lemma \ref{comparison Q_u^+}.
\end{proof}

We denote by $L_u^+$ the restriction
 of $B_u$ to ${\rm{dom}}(L_u^+)$, defined as
 $$
 {\rm{dom}}(L_u^+) := \{ g \in H^{1/2}_+ \, : \,  
 B_u g \in H_+\} \, .
 $$
 We view $L^+_u$ as an unbounded linear operator on $H_+$
 and write
 $L^+_u : {\rm{dom}}( L_u^+) \to H_+$.
 \begin{lemma}\label{operator L^+_u}
 For any $ u \in H^{-s}_{r,0}$ with $0 \le s < 1/2$,
 the following holds:\\
 $(i)$ ${\rm{dom}}(L_u^+)$ is a dense subspace of $H^{1/2}_+$
 and hence of $H_+$.\\
 $(ii)$ $L_u^+ : {\rm{dom}}( L_u^+) \to H_+$ is bijective
 and the right inverse of $L^+_u$,
 $(L^+_u)^{-1} : H_+ \to H_+$,
 is compact. Hence $L_u^+$ has discrete spectrum.\\
 $(iii)$ $(L^+_u)^{-1}$ is symmetric and $L^+_u$ is self-adjoint
 and positive.
 \end{lemma}
 \begin{proof}
 $(i)$ Since $H_+$ is a dense subspace of $H^{-1/2}_+$
 and $B_u^{-1} : H^{-1/2}_+ \to H^{1/2}_+$ 
 is a linear isomorphism, 
 ${\rm{dom}}(L_u^+) = B_u^{-1}(H_+)$ is a dense subspace
 of $H^{1/2}_+$, and hence also of $H_+$.\\
 $(ii)$ Since $L^+_u$ is the restriction of 
 the linear isomorphism $B_u$,
 it is one-to-one. By the definition of $L^+_u$, it is onto.
 The right inverse of $L_u^+$, denoted by $(L_u^+)^{-1}$, 
 is given by the composition
 $\iota \circ B_u^{-1}|_{H_+}$, where
 $\iota: H^{1/2}_+ \to H_+$ is the standard embedding
 which by Sobolev's embedding theorem is compact.
 It then follows that $(L_u^+)^{-1}: H_+ \to H_+$
 is compact as well.\\
 $(iii)$ For any $f, g \in H_+$
 $$
 \langle (L_u^+)^{-1} f | g \rangle =
 \langle A_u^{-1}\langle D \rangle ^{-1}  f | g \rangle =
 \langle A_u^{-1}\langle D  \rangle ^{-1}  f | 
 \langle D  \rangle ^{-1} g \rangle_{1/2} \, .
 $$
 By Lemma \ref{Lax Milgram}, $A_u^{-1}$ is symmetric with respect to the
 $H^{1/2}_+-$inner product. Hence 
 $$
 \langle (L_u^+)^{-1} f | g \rangle =
 \langle \langle D  \rangle ^{-1}  f | 
 A_u^{-1} \langle D \rangle ^{-1} g \rangle_{1/2}
 =  \langle f | (L_u^+)^{-1} g \rangle\, ,
 $$
 showing that $(L_u^+)^{-1}$ is symmetric.
 Since in addition, $(L_u^+)^{-1}$ is bounded
 it is also self-adjoint.
 By Lemma \ref{comparison Q_u^+} it then follows that
 $$
  \langle L_u^+ f | f \rangle = 
 \langle  \langle D \rangle A_u f | f \rangle=
 \langle A_u f | f \rangle_{1/2}=
 Q^+_u(f, f) \ge \frac{1}{2} \| f \|_{1/2}^2\, ,
 $$
 implying that $ L_u^+$ is a positive operator.
 \end{proof}
 We now define for any $ u \in H^{-s}_{r,0}$ 
 with $0 \le s < 1/2$, the operator $L_u$ 
 as a linear operator with domain 
 ${\rm{dom}}(L_u) := {\rm{dom}}(L_u^+)$
 by setting
 $$
 L_u := L_u^+ - \big( 1 + \eta_s(\|u\|_s) \big) :
 {\rm{dom}}(L_u) \to H_+ \, .
 $$
 Lemma \ref{operator L^+_u} yields the following
 \begin{corollary}\label{definition L_u}
 For any $ u \in H^{-s}_{r,0}$ with $0 \le s < 1/2$,
 the operator $L_u : {\rm{dom}}(L_u) \to H_+$
 is densely defined, self-adjoint, bounded from below,
 and has discrete spectrum. It thus admits an $L^2-$normalized
 basis of eigenfunctions, contained in ${\rm{dom}}(L_u)$
 and hence in $H^{1/2}_+$.
 \end{corollary}
 \begin{remark}\label{symmetry of B_u}
 Let $ u \in H^{-s}_{r,0}$ with $0 \le s < 1/2$ be given.
 Since ${\rm{dom}}(L^+_u)$ is dense in $H^{1/2}_+$
 and $L^+_u$ is the restriction of 
 $B_u : H^{1/2}_+ \to H^{-1/2}_+$ to ${\rm{dom}}(L^+_u)$,
 the symmetry
 $$
 \langle L_u^+  f | g \rangle = \langle f | L_u^+  g \rangle 
 \, , \quad \forall f, g \in \rm{dom}(L^+_u)
 $$
 can be extended by a straightforward density argument
 as follows
 $$
  \langle B_u  f | g \rangle = \langle f | B_u  g \rangle 
 \, , \quad \forall f, g \in H^{1/2}_+ \, .
 $$
 \end{remark}
Note that for any $f, g \in H^{1/2}_+$, 
$\langle B_u  f | g \rangle =
\langle \langle D \rangle A_u f | g \rangle =
\langle A_u f | g \rangle_{1/2}$ 
and hence by \eqref{Q_u^+},
$$
\langle B_u  f | g \rangle =
Q_u^+(f, g) = 
\langle Df - T_uf + (1 + \eta_s(\|u\|_{-s}))f | g \rangle \, ,
$$ 
yielding the following identity in $H^{-1/2}_+$,
\begin{equation}\label{identity B_u}
B_u  f = Df - T_uf + (1 + \eta_s(\|u\|_{-s})) f\, ,
\quad \forall f \in H^{1/2}_+ \, .
\end{equation}
 Given $ u \in H^{-s}_{r,0}$ with $0 \le s < 1/2$, 
 let us consider the restriction of $B_u$ to $H^{1-s}_+$.
 \begin{lemma}\label{restriction of B_u to H^{1-s}}
 Let $ u \in H^{-s}_{r,0}$ with $0 \le s < 1/2$.
 Then $B_u(H^{1-s}_+) = H^{-s}_+$ and the 
 restriction
 $B_{u; 1-s}  :=B_u|_{H^{1-s}_+} : H^{1-s}_+ \to H^{-s}_+$
 is a linear isomorphism. The operator norm of
 $B_u|_{H^{1-s}_+}$ 
 and the one of its inverse are bounded uniformly
 on bounded subsets of elements $u \in H^{-s}_{r,0}$.
 \end{lemma}
 \begin{proof}
 Since $1-s>1/2$, $H^{1-s}$ acts by multiplication on itself and on $L^2$, hence  by interpolation on $H^r$ for $0\leq r\leq 1-s$. 
 By duality, it also acts on $H^{-r}$, in particular 
 with $r=s$. This implies that  $B_u|_{H_+^{1-s}} : H_+^{1-s} \to H_+^{-s}$ is bounded.
 Being the restriction of an injective operator, it is
 injective as well. Let us prove that $B_u|_{H_+^{1-s}}$ has
 $H_+^{-s}$ as its image. To this end consider an arbitrary
 element $h \in H_+^{-s}$. We need to show that the
 solution $f\in H^{1/2}_+$ of $B_u f = h$ 
 is actually in $H^{1-s}_+$. Write
 \begin{equation}\label{identity for f}
 Df  = h +  (1 + \eta_s(\|u\|_{-s})) f + T_uf\, .
 \end{equation}
 Note that $h +  (1 + \eta_s(\|u\|_{-s})) f$ is in $H^{-s}_+$
 and it remains to study $T_uf$.
 By Lemma \ref{estimate for T_u f}$(i)$ one infers
 that for any $g \in H^\sigma_+$, with $\sigma =(1/2+s)/2$, 
 $$
| \langle T_u f | g \rangle | = 
| \langle u  | g \overline f \rangle | \le 
\|u \|_{-s} \|g \overline f\|_{s} \le 
C_{1,s}^2 \|u\|_{-s} \| g \|_{\sigma} \| f \|_{\sigma}\, ,
 $$
 implying that $T_u f \in H^{-\sigma}_+$ and hence by
 \eqref{identity for f}, $f \in H^{1-\sigma}$.
 Since $1-\sigma >1/2$,  we argue as at the beginning of the proof 
to infer that $T_u f \in H^{-s}_+$. Thus applying
 \eqref{identity for f} once more  we conclude that 
  $f \in H^{1-s}_+$. This shows that 
  $B_u|_{H_+^{1-s}} : H_+^{1-s} \to H_+^{-s}$ is onto.
  Going through the arguments of the proof one verifies
  that the operator norm of $B_u|_{H^{1-s}_+}$ 
 and the one of its inverse are bounded uniformly
 on bounded subsets of elements $u \in H^{-s}_{r,0}$.
  This completes the proof of the lemma.
 \end{proof}
 Lemma \ref{restriction of B_u to H^{1-s}} has
 the following important 
 \begin{corollary}
 For any $ u \in H^{-s}_{r,0}$ with $0 \le s < 1/2$,
 ${\rm{dom}}(L_u^+) \subset H^{1-s}_+$. In particular,
 any eigenfunction of $L_u^+$ (and hence of $L_u$)
  is in $H^{1-s}_+$.
 \end{corollary}
 \begin{proof}
 Since $H_+ \subset H^{-s}_+,$ one has 
 $B_u^{-1}(H_+) \subset B_u^{-1}(H^{-s}_+)$
 and hence by Lemma \ref{restriction of B_u to H^{1-s}},
${\rm{dom}}(L_u^+) = B_u^{-1}(H_+) 
\subset  H^{1-s}_+$.
\end{proof}
With the results obtained so far, it is straightforward 
to verify that many of the results of \cite{GK}
extend to the case where $u \in H^{-s}_{r,0}$.
More precisely, let $u \in H^{-s}_{r,0}$ with $0 \le s < 1/2$.
We already know that the spectrum of $L_u$ is discrete, 
bounded from below, and real. When listed 
in increasing order and with their multiplicities,
the eigenvalues of $L_u$ satisfy
$\lambda_0 \le \lambda_1 \le \lambda_2 \le \cdots $.
Arguing as in the proof of \cite[Proposition 2.1]{GK}, 
one verifies that $\lambda_n \ge \lambda_{n-1}+ 1$, $n \ge 1$,
and following \cite[(2.10)]{GK} we define
$$
\gamma_n(u) := \lambda_n - \lambda_{n-1} - 1 \ge 0\, .
$$
It then follows that for any $n \ge 1$,
$$
\lambda_n  = n + \lambda_0 + 
\sum_{k=1}^n\gamma_k \ge n + \lambda_0\, .
$$
Since \cite[Lemma 2.1, Lemma 2.2]{GK}
continue to hold for $u \in H^{-s}_{r,0}$, 
we can introduce eigenfunctions $f_n(x, u)$
of $L_u$, corresponding to the eigenvalues $\lambda_n$, which
are normalized as in \cite[Definition 2.1]{GK}. 
The identities \cite[(2.13)]{GK} continue to hold, 
\begin{equation}\label{formula coeff of u}
\lambda_n \langle 1 | f_n \rangle = - \langle u | f_n \rangle
\end{equation}
as does \cite[Lemma 2.4]{GK}, stating that for any $n \ge 1$
\begin{equation}\label{vanishing gamma}
\gamma_n = 0  \quad 
\text{ if and only if } \quad \langle 1 | f_n \rangle = 0 \ .
\end{equation}
Furthermore, the definition \cite[(3.1)]{GK} of the
generating function $\mathcal H_\lambda(u)$
extends to the case where $u \in H^{-s}_{r,0}$
with $0 < s < 1/2$,
\begin{equation}\label{definition generation fucntion}
 \mathcal H_\lambda: H^{-s}_{r,0} \to \C, \, u \mapsto 
\langle (L_u + \lambda)^{-1} 1 | 1 \rangle \, .
\end{equation}
and so do the identity \cite[(3.2)]{GK},
the product representation of $\mathcal H_\lambda(u)$,
stated in \cite[Proposition 3.1(i)]{GK}, 
\begin{equation}\label{generating function for -s}
\mathcal H_\lambda (u)=\frac{1}{\lambda_0+\lambda}\prod_{n=1}^\infty \big(1-\frac{\gamma_n}{\lambda_n+\lambda}   \big) \, ,
\end{equation}
and the one for $|\langle 1 | f_n \rangle |^2$,
$n \ge 1$, given in \cite[Corollary 3.1]{GK}, 
 \begin{equation}\label{product for kappa for -s}
 | \langle 1 | f_n \rangle |^2 = \gamma_n \kappa_n\, , \quad
\kappa_n = \frac{1}{\lambda_n - \lambda_0}
\prod_{p \ne n} (1 - \frac{\gamma_p}{\lambda_p - \lambda_n})
\, .
\end{equation}
The product representation 
\eqref{generating function for -s} then
yields the identity (cf. \cite[Proposition 3.1(ii)]{GK}
and its proof),
\begin{equation}\label{identity for lambda_0}
-\lambda_0(u) = \sum_{n=1}^\infty \gamma_n(u) \,.
\end{equation}
Since $\gamma_n(u) \ge 0$ for any $n \ge 1,$ one infers that
for any $u \in H^{-s}_{r,0}$ with $0 \le s < 1/2$,
the sequence
$(\gamma_n(u))_{n \ge 1}$ is in 
$\ell^1_+ \equiv \ell^1(\N , \R)$ and
\begin{equation}\label{formula for lambda_n}
\lambda_n(u) = n - \sum_{k = n+1}^\infty \gamma_k(u)
\le n\, .
\end{equation}
By \eqref{Q_u^+}, Lemma \ref{comparison Q_u^+}, and \eqref{identity B_u}, we  infer
that $- \lambda_0 \le \frac 12 + \eta_s(\|u\|_{-s})$, 
yielding, when combined with \eqref{identity for lambda_0} 
and \eqref{formula for lambda_n}, the estimate
\begin{equation}\label{sandwich lambda_n}
 n -  \frac 12 - \eta_s(\|u\|_{-s})  \le \lambda_n(u) \le n \, ,
 \quad \forall n \ge 0\, .
\end{equation}
 
In a next step we want consider the linear isomorphism 
 $$
 B_{u; 1-s} = B_u|_{H_+^{1-s}} : H^{1-s}_+ \to H^{-s}_+
 $$ 
 on the scale of Sobolev spaces. By duality,
 $B_{u; 1 - s}$ extends as a bounded linear isomorphism,
 $B_{u,s}:  H^{s}_+ \to H^{-1+s}_+ $ and hence by
 complex interpolation, for any $s \le t \le 1-s$, 
 the restriction of $B_{u,s}$ to $H^{t}_+$
 gives also rise to a bounded linear isomorphism,
 $B_{u; t}: H^{t}_+ \to H^{-1 +t}_+$.
 All these operators satisfy the same bound as $B_{u; 1 - s}$
 (cf. Lemma \ref{restriction of B_u to H^{1-s}}).
 To state our next result, it is convenient to introduce 
 the notation $\N_0 := \Z_{\ge 0}$. Recall that 
 $h^t(\N_0) = h^t(\N_0, \C)$, $t \in \R$,
 and that we write $\ell^2(\N_0)$ instead of
 $h^0(\N_0)$.
 \begin{lemma}\label{basis on Sobolev scale}
 Let $ u \in H^{-s}_{r,0}$ with $0 \le s < 1/2$
 and let $(f_n)_{n \ge 0}$ be the basis of $L^2_+,$
consisting of eigenfunctions of $L_u$ with
$f_n$, $n \ge 0$,
 corresponding to the eigenvalue $\lambda_n$ and normalized
 as in \cite[Definition 2.1]{GK}. Then for any 
 $-1 + s \le t \le 1-s,$
 $$
 K_{u;t} : H^t_+ \to h^t(\N_0)\, , \, 
 f \mapsto (\langle f | f_n \rangle)_{n \ge 0}
 $$
 is a linear isomorphism. In particular, for 
 $f = \Pi u \in H^{-s}_+$, one obtains that
 $(\langle \Pi u | f_n \rangle)_{n \ge 0} 
 \in h^{-s}(\N_0)$.
 The operator norm of $K_{u;t}$ and the one of its
 inverse can be uniformly bounded 
 for $-1 + s \le t \le 1-s$ and for  $u$ in a bounded subset of $ \in H^{-s}_{r,0}$.
 \end{lemma}
 \begin{proof} We claim that the sequence $(\tilde f_n)_{n\ge 0}$, defined by
 $$
 \tilde f_n=\frac{f_n}{(\lambda_n+1+\eta_s(\Vert u\Vert_{-s}))^{1/2}} \, ,
 $$
 is an orthonormal basis of the Hilbert space $H^{1/2}_+$, endowed with the inner product $Q_u^+$. 
 Indeed, for any $n \ge 0$ and any $g\in H^{1/2}_+$, one has
 $$Q_u^+(\tilde f_n, g) =\langle L_u^+\tilde f_n\vert g\rangle =(\lambda_n+1+\eta_s(\Vert u\Vert_{-s}))^{1/2}\langle f_n\vert g\rangle \ .$$
As a consequence, for any $n, m \ge 0$, $Q_u^+(\tilde f_n, \tilde f_m) = \delta_{nm}$ and 
 the orthogonal complement of the subspace of $H^{1/2}_+$, spanned by  $(\tilde f_n)_{n\ge 0}$, is the trivial vector space $\{ 0\} $,
showing that $(\tilde f_n)_{n\ge 0}$ is an orthonormal basis of $H^{1/2}_+$.
In view of \eqref{sandwich lambda_n}, we then conclude that
  $$
 K_{u;1/2} : H^{1/2}_+ \to h^{1/2}(\N_0)\, , \, 
 f \mapsto (\langle f | f_n \rangle )_{n \ge 0}
  $$ 
  is a linear isomorphism.
  Its inverse is given by
  $$
  K_{u;1/2}^{-1} : h^{1/2}(\N_0) \to  H^{1/2}_+\, , \,
  (z_n)_{n \ge 0} \mapsto f:= \sum_{n=0}^\infty z_n f_n
  \ .$$
 By interpolation we infer that 
  for any $0 \le t \le 1/2$,
  $ K_{u;t} : H^{t}_+ \to h^{t}(\N_0)$
  is a linear isomorphism. Taking the transpose of
  $K_{u;t}^{-1}$ it then follows that for any 
  $0 \le t \le 1/2$,
  $$
   K_{u;-t} : H^{-t}_+ \to h^{-t}(\N_0)\, ,\,
   f \mapsto (\langle f | f_n \rangle )_{n \ge 0}\, ,
  $$
 is also a linear isomorphism. 
 It remains to discuss the remaining range of $t,$
 stated in the lemma.
 By Lemma \ref{restriction of B_u to H^{1-s}}, 
 the restriction of $B_u^{-1}$ to $H^{-s}_+$ 
 gives rise to a linear isomorphism
 $
 B_{u; 1-s}^{-1} :  H^{-s}_+ \to  H^{1-s}_+ \,.
 $
 For any $f \in H^{-s}_+$, one then has
 $$
 B_{u; 1-s}^{-1} f  = \sum_{n=0}^\infty 
 \frac{\langle f | f_n \rangle}
 {\lambda_n + 1 + \eta_s(\|u\|_{-s})} f_n \,.
 $$
 Since by our considerations above, 
 $(\langle f | f_n \rangle )_{n \ge 0} \in h^{-s}(\N_0)$
 one concludes that the sequence
 $\big( \frac{\langle f | f_n \rangle}
 {\lambda_n + 1 + \eta_s(\|u\|_{-s})} \big)_{n \ge 0}$
 is in $h^{1-s}(\N_0)$.
 Conversely, assume that $(z_n)_{n \ge 0} \in h^{1-s}(\N_0)$.
Then 
$\big((\lambda_n + 1 + \eta_s(\|u\|_{-s}))z_n  \big)_{n \ge 0}$ 
is in $h^{-s}(\N_0)$. Hence by the considerations above
on $K_{u; -s}$, there exists $g \in H^{-s}_+$
so that
$$
\langle g | f_n \rangle = 
(\lambda_n + 1 + \eta_s(\|u\|_{-s}))z_n\, , 
\quad \forall n \ge 0\, .
$$
Hence 
$$
g = \sum_{n=0}^\infty 
z_n (\lambda_n + 1 + \eta_s(\|u\|_{-s})) f_n 
= \sum_{n=0}^\infty 
z_n B_{u} f_n
$$ 
and $f:= B_{u}^{-1}g$ is in $H^{1-s}_+$ and satisfies 
$f = \sum_{n=0}^\infty z_n f_n$. Altogether we have thus
proved that
$$
K_{u;1 -s} : H^{1 -s}_+ \to h^{1 -s}(\N_0)\, ,\,
f \mapsto (\langle f | f_n \rangle )_{n \ge 0}\, ,
  $$
is a linear isomorphism. Interpolating between
$K_{u;-s}$ and $K_{u;1 -s}$ and between the adjoints
of their inverses shows that
for any $-1 + s \le t \le 1-s,$
 $$
 K_{u;t} : H^t \to h^t(\N_0)\, , \, 
 f \mapsto (\langle f | f_n \rangle)_{n \ge 0}
 $$
 is a linear isomorphism. Going through the arguments
 of the proof one verifies that 
 the operator norm of $K_{u;t}$ and the one of its
 inverse can be uniformly bounded 
 for $-1 + s \le t \le 1-s$ and for bounded subsets
 of elements $u \in H^{-s}_{r,0}$.
 \end{proof}
 
\smallskip

With these preparations done, we can now prove
Proposition \ref{extension Phi, part 1}$(i)$.

\smallskip
 \noindent
{\em Proof of Proposition \ref{extension Phi, part 1}$(i)$. }
 Let $ u \in H^{-s}_{r,0}$ with $0 \le s < 1/2$.
 By \eqref{product for kappa for -s}, one has 
 for any $n \ge 1$, 
 $$
 | \langle 1 | f_n \rangle |^2 = \gamma_n \kappa_n\, , \qquad
\kappa_n = \frac{1}{\lambda_n - \lambda_0}
\prod_{p \ne n} (1 - \frac{\gamma_p}{\lambda_p - \lambda_n})
\, .
$$
Note that the infinite product is absolutely
convergent since the sequence $(\gamma_n(u))_{n \ge 1}$ 
is in $\ell^1_+$ (cf. \eqref{identity for lambda_0}). 
Furthermore, since
$$
1 - \frac{\gamma_p}{\lambda_p - \lambda_n}=
\frac{\lambda_{p-1} + 1 - \lambda_n}{\lambda_p - \lambda_n}
> 0 \, , \quad \forall p \ne n
$$
it follows that $\kappa_n > 0$ for any $n \ge 1$. 
Hence, the formula \cite[(4.1)]{GK} 
of the Birkhoff coordinates $\zeta_n(u)$, $n \ge 1$, 
defined for $u \in L^2_{r,0}$,
\begin{equation}\label{def coo for -s}
\zeta_n(u) = \frac{1}{\sqrt{\kappa_n(u)}} 
\langle 1 | f_n(\cdot, u) \rangle \, ,
\end{equation}
extends to $H^{-s}_{r,0}$. 
By \eqref{formula coeff of u} 
one has (cf. also \cite[(2.13)]{GK})
$$
\lambda_n \langle 1 | f_n \rangle 
= - \langle u | f_n \rangle
= - \langle \Pi u | f_n \rangle \, .
$$
 Since by Lemma \ref{basis on Sobolev scale},
 $(\langle \Pi u | f_n \rangle)_{n \ge 0} 
 \in h^{-s}(\N_0)$
and by \eqref{sandwich lambda_n}
$$
n -  \frac 12 - \eta_s(\|u\|_{-s})  \le \lambda_n(u) \le n \, , \quad \forall n \ge 0 \, ,
$$
 one concludes that 
 $$
( \langle 1 | f_n \rangle)_{n \ge 1} \in h^{1-s}_+\, , \qquad
\kappa_n^{-1/2} = \sqrt{n} + o(1)
 $$ 
 and hence $(\zeta_n(u))_{n \ge 1} \in h^{1/2-s}_+$. 
In summary, we have proved that for any $0 < s < 1/2,$
the Birkhoff map $\Phi: L^2_{r,0} \to h^{1/2}_+$ of
Theorem \ref{main result} extends to a map 
$$
H^{-s}_{r,0} \to h^{1/2 - s}_+ \, , \,
u \mapsto  (\zeta_n(u))_{n \ge 1}\, , 
$$
 which we again denote by $\Phi$. Going through the arguments
 of the proof one verifies that $\Phi$ maps bounded subsets of $H^{-s}_{r,0}$ into bounded subsets of $h^{1/2 - s}_+$ . 
 \hfill $\square$

\smallskip

To show Proposition \ref{extension Phi, part 1}$(ii)$
we first need to prove some additional auxilary results.
By \eqref{definition generation fucntion},
the generating function is defined as
$$
 \mathcal H_\lambda: H^{-s}_{r,0} \to \C, \, u \mapsto 
\langle (L_u + \lambda)^{-1} 1 | 1 \rangle \, .
$$
For any given $u \in H^{-s}_{r,0}$, $\mathcal H_\lambda(u)$
is a meromorphic function
in $\lambda \in \C$ with possible poles at the eigenvalues of 
$L_u$ and satisfies (cf.  \eqref{generating function for -s})
\begin{equation}\label{product}
 \mathcal H_\lambda (u)=\frac{1}{\lambda_0+\lambda}\prod_{n=1}^\infty \big(1-\frac{\gamma_n}{\lambda_n+\lambda}   \big)\, . 
\end{equation}

\begin{lemma}\label{gamma_n for -s}
For any $0 \le s < 1/2,$ the following holds:\\
$(i)$ For any $\lambda \in \C \setminus \R,$ 
$\mathcal H_\lambda: H^{-s}_{r,0} \to \C$ is 
sequentially weakly continuous.\\
$(ii)$ $(\sqrt{\gamma_n})_{n \ge 1} : H^{-s}_{r,0} \to 
h^{1/2 -s}_+$
is sequentially weakly continuous. 
In particular, for any $n \ge 0,$
$\lambda_n: H^{-s}_{r,0} \to \R$ is 
sequentially weakly continuous.
\end{lemma}
\begin{proof}
$(i)$ 
Let $(u^{(k)})_{k \ge 1}$ be a sequence in 
$H^{-s}_{r,0}$ with
$u^{(k)}\rightharpoonup u$ weakly in $H^{-s}_{r,0}$ 
as $k \to \infty$. By the definition of $\zeta_n(u)$
(cf. \eqref{product for kappa for -s} -- 
\eqref{def coo for -s})
one has $| \zeta_n(u) |^2 = \gamma_n(u)$.
Since by 
Proposition \ref{extension Phi, part 1}$(i)$,
$\Phi$ maps bounded subsets of $H^{-s}_{r,0}$
to bounded subsets of $h^{1/2 -s}_+$,
there exists $M > 0$ so that for any $k \ge 1$
$$
\|u\|, \,  \|u^{(k)}\| \le M\, , \qquad
\sum_{n=1}^\infty n^{1-2s} \gamma_n(u), \,\,
\sum_{n=1}^\infty n^{1-2s} \gamma_n(u^{(k)}) \le M\, .
$$
By passing to a subsequence, if needed, we may assume that
\begin{equation}\label{lim gamma_n for -s}
\big( \gamma_n(u^{(k)})^{1/2}\big)_{n \ge 1}
\rightharpoonup (\rho_n^{1/2})_{n \ge 1}
\end{equation}
weakly in $h^{1/2 -s}(\N, \R)$ 
where $\rho_n \ge 0$ for any $n \ge 1$.
It then follows that 
$\big(\gamma_n(u^{(k)})\big)_{n \ge 1}
\to (\rho_n)_{n \ge 1}$
 strongly in $\ell^1(\N, \R)$. Define
 $$
 \nu_n:= n - \sum_{p = n+1}^{\infty} \rho_p\,, \quad
 \forall n \ge 0\, .
 $$
Then for any $n \ge 1$, $\nu_n = \nu_{n-1} + 1 + \rho_n$
and $\lambda_n(u^{(k)}) \to \nu_n$ uniformly in $n \ge 0.$
Since $L_{u^{(k)}} \ge \lambda_0(u^{(k)})$
we infer that there exists $c > |-\nu_0 + 1|$
so that for any $k \ge 1$ and $\lambda \ge c$,
$$
L_{u^{(k)}} + \lambda : H^{1-s}_+ \to H^{-s}_+
$$
is a linear isomorphism whose inverse is bounded uniformly in $k$. Therefore
$$
w_\lambda^{(k)}:= 
\big(L_{u^{(k)}} + \lambda  \big)^{-1}[1]\, , \quad
\forall k \ge 1\, ,
$$
is a well-defined, bounded sequence in $H^{1-s}_+.$
Let us choose an arbitrary countable subset 
$\Lambda \subset [c, \infty)$ with one cluster point. By a diagonal procedure, we extract a subsequence
of $(w_\lambda^{(k)})_{k \ge 1}$, again denoted by
$(w_\lambda^{(k)})_{k \ge 1}$, so that for every 
$\lambda \in \Lambda,$ the sequence $(w_\lambda^{(k)})$
converges weakly in $H^{1-s}_+$ to some element 
$v_\lambda \in H^{1-s}_+$. By Rellich's theorem 
$$
\big(L_{u^{(k)}} + \lambda  \big) w_\lambda^{(k)}
\rightharpoonup 
\big(L_{u} + \lambda  \big) v_{\lambda}
$$
weakly in $H^{-s}_+$ as $k \to \infty.$
Since by definition, 
$\big(L_{u^{(k)}} + \lambda  \big) 
w_\lambda^{(k)} = 1$ for any $k \ge 1,$ it follows that
for any $\lambda \in \Lambda,$
$\big(L_{u} + \lambda  \big) 
v_\lambda = 1$ and thus by the definition of 
the generating function
$$
\mathcal H_\lambda (u^{(k)}) =
\langle w_\lambda^{(k)} | 1 \rangle \to
\langle v_\lambda  | 1 \rangle = 
\mathcal H_\lambda (u)\, , \quad
\forall \lambda \in \Lambda \, .
$$
Since $\mathcal H_\lambda (u^{(k)}) $ and 
$\mathcal H_\lambda (u)$ are meromorphic functions
whose poles are on the real axis, it follows that
the convergence holds for any 
$\lambda \in \C \setminus \R$. This proves
item $(i)$. \\
$(ii)$ We apply item $(i)$ (and its proof) as follows.
As mentioned above, 
$\lambda_n(u^{(k)}) \to \rho_n$, uniformly in $n \ge 0$.
By the proof of item $(i)$ one has for any $c \le \lambda < \infty,$
$$
\mathcal H_\lambda (u^{(k)})  \to 
\frac{1}{\nu_0+\lambda}\prod_{n=1}^\infty \big(1-\frac{\rho_n}{\nu_n+\lambda}   \big)\,  
$$
and we conclude that for any $\lambda \in \Lambda$
$$
\frac{1}{\lambda_0(u)+\lambda}\prod_{n=1}^\infty 
\big(1-\frac{\gamma_n(u)}{\lambda_n(u) +\lambda}   \big)
= \mathcal H_\lambda (u) =
\frac{1}{\nu_0+\lambda}\prod_{n=1}^\infty \big(1-\frac{\rho_n}{\nu_n+\lambda}   \big) .
$$
Since $\mathcal H_\lambda (u)$ and
infinite product are meromorphic functions
in $\lambda$, the functions are equal. In particular,
they have the same zeroes and the same poles.
Since the sequences 
$(\lambda_n(u))_{n \ge 0}$ and 
$(\nu_n(u))_{n \ge 0}$
are both listed in increasing order it follows that 
$\lambda_n(u) = \nu_n$ for any $n \ge 0$,
implying that for any $n \ge 1$, 
$$
\gamma_n(u) = 
\lambda_n(u) - \lambda_{n-1}(u) -1 = 
\nu_n - \nu_{n-1} -1 = \rho_n\, .
$$
By \eqref{lim gamma_n for -s} we then conclude that
$$
\big( \gamma_n(u^{(k)})^{1/2}\big)_{n \ge 1}
\rightharpoonup (\gamma_n(u)^{1/2})_{n \ge 1}
$$
weakly in $h^{1/2 -s}(\N, \R)$. 
\end{proof}

\begin{corollary}\label{kappa_n for neg s}
For any $0 \le s < 1/2$ and $n \ge 1$, the functional 
$\kappa_n : H^{-s}_{r,0} \to \R$,
introduced in \eqref{product for kappa for -s},
is sequentially weakly continuous.
\end{corollary}
\begin{proof}
Let $(u^{(k)})_{k \ge 1}$ be a sequence in 
$H^{-s}_{r,0}$ with
$u^{(k)}\rightharpoonup u$ weakly in $H^{-s}_{r,0}$ 
as $k \to \infty$. By \eqref{formula for lambda_n},
one has for any $p < n,$
$$
\lambda_p(u^{(k)}) - \lambda_n(u^{(k)}) =
p-n - \sum_{j = p+1}^n \gamma_j(u^{(k)})
$$
whereas for $p > n$
$$
\lambda_p(u^{(k)}) - \lambda_n(u^{(k)}) =
p-n + \sum_{j = n+1}^p \gamma_j(u^{(k)})\, .
$$
By Lemma \ref{gamma_n for -s}, one then concludes that
$$
\lim_{k \to \infty} \big( \lambda_p(u^{(k)}) - \lambda_n(u^{(k)})\big) - 
\big( \lambda_p(u) - \lambda_n(u)\big) =0
$$
uniformly in $p, n \ge 0$. By the product formula 
\eqref{product for kappa for -s}
for $\kappa_n$ it then follows that for any $n \ge 1$, one has
$\lim_{k \to \infty} \kappa_n(u^{(k)}) = \kappa_n(u)$.
\end{proof} 

Furthermore, we need to prove the following lemma
concerning the eigenfunctions 
$f_n(\cdot, u)$, $n \ge 0$, of $L_u$.
\begin{lemma}\label{f_n bounded for s neg}
Given $0 \le s < 1/2$, $M > 0$ and $n \ge 0,$
there exists a constant $C_{s, M, n} \ge 1$ so that for any 
$u \in H^{-s}_{r,0}$ with $\|u \|_{-s} \le M$ and any $n \ge 0$
\begin{equation}\label{estimate f_n final}
\| f_n(\cdot , u)) \|_{1-s} \le C_{s, M, n} \, .
\end{equation}
\end{lemma}
\begin{proof}
By the normalisation of $f_n$,
 $\|f_n\| = 1$. Since $f_n$ is an eigenfunction, corresponding
 to the eigenvalue $\lambda_n,$ one has
 $$
  -i\partial_x f_n = L_u f_n + T_u f_n = 
 \lambda_n f_n + T_u f_n \, ,
 $$
implying that
\begin{equation}\label{estimate of f_n for neg s}
\| \partial_x f_n \|_{-s} \le |\lambda_n| + \|T_uf_n\|_{-s}
\, .
\end{equation}
Note that by the estimates \eqref{sandwich lambda_n}, 
\begin{equation}\label{estimate | lambda_n |}
|\lambda_n| \le \max\{ n , |\lambda_0| \}
\le n + |\lambda_0|\, , \quad
|\lambda_0 | \le 1+ \eta_s(\| u \|_{-s}) \, 
\end{equation}
where $\eta_s(\| u \|_{-s})$ is given by \eqref{def eta_s}.
Furthermore, since $\sigma = (1/2 +s)/2$ (cf. \eqref{para}) one has
$1-s>1- \sigma >1/2$, implying that $H_+^{1-\sigma}$ acts on $H_+^t$ 
for every $t$ in the open interval $(-(1-\sigma),1-\sigma)$. Hence 
\begin{equation}\label{estimate T_u f_n}
\| T_u f_n\|_{-s} \le C_{s}
\|u \|_{-s} \| f_n \|_{1-\sigma} \, .
\end{equation}
Using interpolation and Young's inequality
(cf. \eqref{interpolation}, \eqref{Young}),
\eqref{estimate T_u f_n} yields an estimate, which together
with \eqref{estimate of f_n for neg s} and 
\eqref{estimate | lambda_n |} leads to the claimed estimate 
\eqref{estimate f_n final}.
\end{proof}
 
With these preparations done, we can now prove
Proposition \ref{extension Phi, part 1}$(ii)$.

 \smallskip
 
 \noindent
{\em Proof of Proposition \ref{extension Phi, part 1}$(ii)$.}
 First we prove that for any $0 \le s < 1/2$,
$\Phi: H^{-s}_{r,0} \to h^{1/2 - s}_+$
is sequentially weakly continuous:
assume that $(u^{(k)})_{k \ge 1}$ is a sequence in 
$H^{-s}_{r,0}$ with
$u^{(k)} \rightharpoonup u$ weakly in $H^{-s}_{r,0}$ 
as $k \to \infty$. Let 
$\zeta^{(k)}:= \Phi(u^{(k)})$ and $\zeta := \Phi(u)$.
Since $(u^{(k)})_{k \ge 1}$ is bounded
in $H^{-s}_{r,0}$
and $\Phi$ maps bounded subsets of $H^{-s}_{r,0}$
to bounded subsets of $h^{1/2 -s}_+$,
the sequence $(\zeta^{(k)})_{k \ge 1}$ is bounded in 
$h^{1/2 - s}_+$. To show that 
$\zeta^{(k)}\rightharpoonup \zeta$ weakly in $h^{1/2}_+$,
it then suffices to prove that for any $n \ge 1,$
$\lim_{k \to \infty} \zeta^{(k)}_n = \zeta_n$.
By the definition of the Birkhoff coordinates 
\eqref{def coo for -s},
$\zeta^{(k)}_n =
\langle 1 | f_n^{(k)} \rangle / (\kappa_n^{(k)})^{1/2}$ 
where
$\kappa_n^{(k)}:= \kappa_n(u^{(k)})$ and 
$f_n^{(k)} := f_n(\cdot, u^{(k)})$.
By Corollary \ref{kappa_n for neg s},
$\lim_{k \to \infty}\kappa_n^{(k)} = \kappa_n$
and by Lemma \ref{f_n bounded for s neg},
saying that for any $n \ge 0$,  $\|f_n\|_{1-s}$
is uniformly bounded on bounded subsets of $H^{-s}_{r,0}$,
$\lim_{k \to \infty}\langle 1 | f_n^{(k)} \rangle =
\langle 1 | f_n \rangle$
where $\kappa_n := \kappa_n(u)$ and
$f_n := f_n(\cdot, u)$. This implies that 
$\lim_{k \to \infty} \zeta^{(k)}_n = \zeta_n$
for any $n \ge 1$.\\
It remains to show that for any $0 < s < 1/2,$
$\Phi : H^{-s}_{r,0} \to h^{1/2-s}_+$
is one-to-one. In the case where
$u \in L^2_{r,0}$,  it was verified
in the proof of \cite[Proposition 4.2]{GK}
that the Fourier coefficients $\widehat u(k)$, $k\ge 1$,
of $u$ can be explicitly expressed in terms of the components 
$\zeta_n(u)$ of the sequence $\zeta(u) = \Phi(u)$.
These formulas continue to hold for $u \in H^{-s}_{r,0}$.
This completes the proof of 
Proposition \ref{extension Phi, part 1}$(ii)$.
\hfill $\square$

\medskip

\section[extension of Birkhoff map. Part 2] {Extension of $\Phi$. Part 2}\label{extension of Phi, part 2}

In this section we prove the second part of 
Theorem \ref{extension Phi}, which we again state as a separate proposition.

\begin{proposition}\label{extension Phi, part 2}
{\sc (Extension of $\Phi$. Part 2)}
For any $0 < s < 1/2,$ the map 
$\Phi: H^{-s}_{r,0} \to h^{1/2-s}_+$  
has the following additional properties:\\
$(i)$ The inverse image of $\Phi$ of any
bounded subset of $h^{1/2-s}_+$ is a bounded
subset in 
$H^{-s}_{r,0}$. \\
$(ii)$ $\Phi$ is onto and the inverse map
$\Phi^{-1}: h^{1/2-s}_+ \to H^{-s}_{r,0}$,
is sequentially weakly continuous.\\
$(iii)$
For any $0 < s < 1/2,$ the Birkhoff map 
$\Phi: H^{-s}_{r,0} \to h^{1/2-s}_+$  
and its inverse 
$\Phi^{-1}: h^{1/2-s}_+ \to H^{-s}_{r,0}$,
are continuous. 
\end{proposition}
\begin{remark}
As mentioned in Remark \ref{weakPhi}, the map
$\Phi: L^2_{r,0} \to h^{1/2}_+$ and its inverse
$\Phi^{-1}: h^{1/2}_+ \to L^2_{r,0}$
are sequentially weakly continuous.
\end{remark}

\noindent
{\em Proof of Proposition \ref{extension Phi, part 2}$(i)$. }
Let $0 < s < 1/2$ and $u \in H^{-s}_{r,0}$.
Recall that by Corollary \ref{definition L_u},
$L_u$ is a self-adjoint operator with domain 
${\rm{dom}}(L_u) \subset H_+$,
has discrete spectrum and is bounded from below. 
Thus $L_u - \lambda_0(u) + 1 \ge 1$ where $\lambda_0(u)$
denotes the smallest eigenvalue of $L_u.$ By the considerations
 in Section \ref{extension of Phi, part 1}
 (cf. Lemma \ref{restriction of B_u to H^{1-s}}),
$L_u$ extends to a bounded operator 
$L_u: H^{1/2}_+ \to H^{-1/2}_+$ and satisfies 
$$
\langle L_u f | f \rangle = 
\langle D f | f \rangle - 
\langle u | f \overline f  \rangle\, , 
\quad \forall f \in H^{1/2}_+ \, .
$$
By Lemma \ref{estimate for T_u f}$(i)$ one has
$ | \langle  u | f \overline f \rangle |
 \le  C_{1,s}^2 \|u\|_{-s} \| f\|_{1/2}^2$ for any $f \in H^{1/2}_+$ and hence
 \begin{align}
 \| f \|^2 & \le 
 \langle ( L_u - \lambda_0(u) + 1) f | f \rangle
 \nonumber \\
 &\le \langle D f | f \rangle + 
 C_{1,s}^2 \|u\|_{-s}\| f\|_{1/2}^2
 + (- \lambda_0(u) + 1) \|f\|^2 \, , \nonumber
 \end{align}
yielding the estimate
$$
 \| f \|^2 \le  
 \langle ( L_u - \lambda_0(u) + 1) f | f \rangle \le
 M_u \|f\|_{1/2}^2
 $$
 where
\begin{equation}\label{def M_u}
 M_u:= C_{1,s}^2 \|u\|_{-s} + (2  - \lambda_0(u)) \,.
\end{equation}
To shorten notation, we will for the remainder of
the proof no longer indicate the dependence of
spectral quantities such as $\lambda_n$ or $\gamma_n$
on $u$ whenever appropriate.
The square root of the operator 
$L_u - \lambda_0 + 1$,
$$
R_u := (L_u - \lambda_0 + 1)^{1/2} : 
H^{1/2}_+ \to H_+\, ,
$$
can then be defined 
in terms of the basis $f_n \equiv f_n(\cdot, u)$,
$n \ge 0$, of eigenfunctions of $L_u$ in a standard way 
as follows: 
By Lemma \ref{basis on Sobolev scale}, any
$f \in H^{1/2}_+$ has an expansion of the form
$f = \sum_{n=0}^\infty \langle f | f_n \rangle f_n$
where $(\langle f | f_n \rangle)_{n \ge 0}$
is a sequence in $h^{1/2}(\N_0)$. $R_u f$ is then defined
as 
$$
R_u f := \sum_{n=0}^\infty 
(\lambda_n - \lambda_0 + 1)^{1/2}
\langle f | f_n \rangle f_n
$$
Since  $(\lambda_n - \lambda_0 + 1)^{1/2} 
\sim \sqrt{n}$ (cf. \eqref{sandwich lambda_n})
one has
$$
\big( (\lambda_n - \lambda_0 + 1)^{1/2}
\langle f | f_n \rangle)^{1/2}\big)_{n \ge 0}
\in \ell^2(\N_0)
$$
implying that $R_u f \in H_+$
(cf. Lemma \ref{basis on Sobolev scale}). 
Note that
$$
\|f\|^2 \le \langle R_u f | R_u f  \rangle
= \langle R_u^2 f | f  \rangle
\le M_u \|f\|^2_{1/2}\, , \quad
\forall f \in H^{1/2}_+\, ,
$$
and that $R_u$ is a positive self-adjoint operator
when viewed as an operator with domain $H^{1/2}_+$, 
acting on $H_+$.
By complex interpolation 
(cf. e.g. \cite[Section 1.4]{Tay}) one then concludes that
for any $0 \le \theta \le 1$
$$
R_u^{\theta} : H^{\theta/2}_+ \to H_+\, , \qquad
\| R_u^{\theta} f\|^2 \le 
M_u^\theta \|f \|^2_{\theta/2} \, , 
\quad \forall f \in H^{\theta/2}_+ \, .
$$
Since by duality, 
$$
R_u^{\theta} : H_+ \to H^{-\theta/2}_+ \, , \qquad
\| R_u^{\theta} g\|_{-\theta /2}^2 \le 
M_u^\theta \|g \|^2 \, , 
\quad \forall g \in H_+ \, ,
$$
one infers, 
using that $R_u^{\theta} : H_+ \to H^{-\theta/2}_+$
is boundedly invertible, that
for any $f \in H^{-\theta /2}_+$,
$$
\| f\|_{-\theta /2}^2 \le 
M_u^\theta \| R_u^{-\theta} f\|^2 \, ,
\qquad
R_u^{-\theta} := (R_u^{\theta})^{-1}\, .
$$
Applying the latter inequality
to $f =\Pi u$ and  $\theta = 2s$
and using that
$\Pi u = \sum_{n =1}^\infty 
\langle \Pi u | f_n\rangle f_n$
and   
$\langle \Pi u | f_n\rangle =
 - \lambda_n  \langle 1 | f_n\rangle$
one sees that
\begin{equation}\label{bound 1 for -s norm of u}
\frac{1}{2} \|u\|^{2}_{-s} = \|\Pi u\|^{2}_{-s}
\le M_u^{2s} \Sigma 
\end{equation}
where
$$
\Sigma:= \sum_{n=1}^\infty \lambda_n^{2}
\big( \lambda_n - \lambda_0+ 1 \big)^{-2s}
 \ |\langle 1\vert f_n\rangle |^2 \ .
 $$
We would like to deduce from 
\eqref{bound 1 for -s norm of u}
an estimate of $\|u\|_{-s}$ in terms of the
 $\gamma_n$'s.
Let us first consider $M_u^{2s}$. 
By \eqref{def M_u} one has
$$
M_u^{2s} = 2^{2s} \max \{ (C_{1,s}^2 \|u\|_{-s})^{2s},  
(2  - \lambda_0(u))^{2s}\} \, ,
$$ 
yielding
\begin{equation}\label{split M_u}
M_u^{2s}
\le (\|u\|_{-s}^2)^s (2 C_{1,s}^2)^{2s}
+ (2(2  - \lambda_0(u)))^{2s} \,.
\end{equation}
Applying Young's inequality 
with $1/p = s$, $1/q = 1 - s$
one obtains
\begin{equation}\label{bound first term of bound 1}
(\|u\|_{-s}^2)^s (2 C_{1,s}^2)^{2s} \Sigma
\le \frac{1}{4} \|u\|_{-s}^2 +
\big( (4 C_{1,s}^2)^{2s} \Sigma \big)^{1/(1-s)} \, ,
\end{equation}
which when combined with \eqref{bound 1 for -s norm of u} 
and \eqref{split M_u}, leads to
$$
\frac{1}{4} \|u\|_{-s}^2 \le 
\big( (4 C_{1,s}^2)^{2s} \Sigma \big)
^{1/(1-s)} 
+ (2(2  - \lambda_0(u)))^{2s} \Sigma \, .
$$
The latter estimate is of the form
\begin{equation}\label{bound 2 of s norm of u}
\|u\|_{-s}^2 \le 
C_{3,s} \Sigma^{1/(1-s)} 
+C_{4,s} (2  - \lambda_0(u))^{2s} \Sigma \, ,
\end{equation}
where $ C_{3,s}, C_{4,s} > 0$ are constants, only
depending on $s$.
Next let us turn to 
$\Sigma = \sum_{n=1}^\infty \lambda_n^{2}
\big( \lambda_n - \lambda_0+ 1 \big)^{-2s}
 \ |\langle 1\vert f_n\rangle |^2$.
Since 
$$
\lambda_n = n - \sum_{k=n+1}^\infty \gamma_k \ ,
\qquad 
\ |\langle 1\vert f_n\rangle |^2 = \gamma_n \kappa_n \, .
$$
and 
\begin{equation}\label{product for kappa}
\kappa_n = \frac{1}{\lambda_n - \lambda_0}
\prod_{p \ne n} (1 - \frac{\gamma_p}{\lambda_p - \lambda_n})\, ,
\end{equation}
the series $\Sigma$ can be expressed in terms of the $\gamma_n$'s. To obtain a bound for $\Sigma$
it remains to estimate the $\kappa_n$'s. Note that
$$
\prod_{p \ne n} 
(1 - \frac{\gamma_p}{\lambda_p - \lambda_n})
\le \prod_{p < n} 
(1 + \frac{\gamma_p}{\lambda_n - \lambda_p})
\le e^{\sum_{p=1}^n \gamma_p} \le e^{- \lambda_0} \, .
$$
Since $ (\lambda_n - \lambda_0)^{-1} = 
(n + \sum_{k=1}^n \gamma_k)^{-1} \le n^{-1}$,
it then follows that
$$
0 < \kappa_n \le  \ \frac{e^{- \lambda_0}}{n}\, , \quad \forall n \ge 1.
$$
Combining the estimates above we get
$$
\Sigma \le e^{-\lambda_0} 
\sum_{n=1}^\infty \lambda_n^{2} n^{-2s -1}  \gamma_n  \ .
$$
By splitting the sum $\Sigma$
into two parts, 
$\Sigma = 
\sum_{ n < - \lambda_0(u) } + 
\sum_{n \ge - \lambda_0(u) }
$ 
and taking into account that $0 \le \lambda_n \le n$ for any 
$n \ge - \lambda_0$ and  $|\lambda_n| \le -\lambda_0 $
for any $1 \le n < - \lambda_0$, one has
$$
\Sigma \le (1 - \lambda_0)^2e^{-\lambda_0} 
\sum_{n = 1}^\infty n^{1 -2s} \gamma_n \, .
$$
Together with the estimate 
\eqref{bound 2 of s norm of u} this shows that
the inverse image  by $\Phi$  
of any bounded subset 
of sequences in $h^{1/2 -s}$ is bounded 
in $H^{-s}_{r,0}$.
\hfill $\square$

\medskip

\noindent
{\em Proof of Proposition \ref{extension Phi, part 2}$(ii)$. }
 First we prove that for any $0 < s < 1/2$,
$\Phi: H^{-s}_{r,0} \to h^{1/2-s}_+$ is onto.
Given $z=(z_n)_{n \ge 1}$ in $h^{1/2-s}_+,$ consider
the sequence $\zeta^{(k)} = (\zeta^{(k)}_n)_{n \ge 1}$,
defined for any $k \ge 1$ by
$$
\zeta^{(k)}_n = z_n\, \,\,\, \forall 1 \le n \le k\, , \qquad
\zeta^{(k)}_n = 0\, \,\,\, \forall n > k\, .
$$
Clearly $\zeta^{(k)} \to z$ strongly in $h^{1/2-s}$.
Since for any $k \ge 1$, $\zeta^{(k)} \in h^{1/2}_+$ 
Theorem \ref{main result} implies that there exists 
a unique element
$u^{(k)}\in L^2_{r,0}$ with $\Phi(u^{(k)})=\zeta^{(k)}$.
By Proposition \ref{extension Phi, part 2}$(i)$, 
$\sup_{k \ge 1}\|u^{(k)} \|_{-s} < \infty$.
Choose a weakly convergent subsequence 
$(u^{(k_j)})_{j \ge 1}$ of $(u^{(k)})_{k \ge 1}$
and denote its weak limit in $H^{-s}_{r,0}$ by $u$.
Since by Proposition \ref{extension Phi, part 1},
$\Phi: H^{-s}_{r,0} \to h^{1/2-s}_+$ is 
sequentially weakly continuous,
$\Phi(u^{(k_j)}) \rightharpoonup \Phi(u)$
weakly in $h^{1/2-s}_+$. On the other hand,
$\Phi(u^{(k_j)}) = \zeta^{(k_j)} \to z$ strongly in $h^{1/2-s}$,
implying that $\Phi(u) = z$.
 This shows that $\Phi$ is onto.\\
 It remains to prove that for any $0 \le s < 1/2$, $\Phi^{-1}$
is sequentially weakly continuous. 
Assume that $(\zeta^{(k)})_{k \ge 1}$ is a sequence in $h^{1/2-s}$,
weakly converging to $\zeta \in h^{1/2-s}$.
Let $u^{(k)}:= \Phi^{-1}(\zeta^{(k)})$.
By Proposition \ref{extension Phi, part 2}$(i)$
(in the case $0 < s < 1/2$) and
Remark \ref{RemarkThm1}$(ii)$ (in the case $s=0$),
 $(u^{(k)})_{k \ge 1}$ is a bounded sequence in
$H^{-s}_{r,0}$ and thus admits a weakly convergent
subsequence $(u^{k_j})_{j \ge 1}$. Denote its
limit in $H^{-s}_{r,0}$ by $u$.
Since by Proposition \ref{extension Phi, part 1}, $\Phi$
is sequentially weakly continuous, $\Phi(u^{k_j}) 
\rightharpoonup \Phi(u)$
weakly in $h^{1/2-s}$. 
On the other hand, by assumption,
$\Phi(u^{k_j}) = \zeta^{(k_j)} \rightharpoonup \zeta$
and hence $u = \Phi^{-1}(\zeta)$
and $u$ is independent
of the chosen subsequence $(u^{k_j})_{j \ge 1}$.
This shows that 
$\Phi^{-1}(\zeta^{(k)}) \rightharpoonup  \Phi^{-1}(\zeta)$
weakly in $H^{-s}_{r,0}$.
\hfill $\square$

\medskip

\noindent
{\em Proof of 
Proposition \ref{extension Phi, part 2}$(iii)$. }
By Proposition \ref{extension Phi, part 1},
$\Phi:  H^{-s}_{r,0} \to h^{1/2-s}_+$ is 
sequentially weakly continuous for any $0 \le s < 1/2$. 
To show that this map is continuous
it then suffices to prove that
the image $\Phi(A)$ of any
relatively compact subset $A$ of $ H^{-s}_{r,0}$ 
 is relatively compact in $h^{1/2-s}_+$.
For any given $\e > 0,$ choose $N \equiv N_\e \ge 1$ and
$R \equiv R_\e > 0$ as in 
Lemma \ref{rel compact subsets}, stated below.
Decompose $u \in A$ as $u = u_N + u_{\bot}$ where
$$
u_N := \sum_{0 <|n| \le N_\e} \widehat u(n) e^{inx} \, ,
\qquad
u_\bot := \sum_{|n| > N_\e} \widehat u(n) e^{inx} \, .
$$
By Lemma \ref{rel compact subsets}, $\|u_N\| < R_\e$
and $\| u_\bot \|_{-s} < \e.$ 
By Lemma \ref{basis on Sobolev scale}, applied
with $\theta = -s$, one has
$$
K_{u; -s}(\Pi u) = K_{u; -s}(\Pi u_N) +
K_{u; -s}(\Pi u_\bot) \in h^{-s}(\N_0)
$$
where  
$K_{u; -s}(\Pi u_N) = K_{u; 0}(\Pi u_N)$
since $\Pi u_N \in H_+$. 
Lemma \ref{basis on Sobolev scale} then implies 
that there exists $C_A > 0$, independent of $u \in A$, 
so that
$$
\| K_{u; 0}(\Pi u_N) \| \le C_A R_\e \, , \qquad
\| K_{u; -s}(\Pi u_\bot) \|_{-s} \le C_A \e \, .
$$
Since $\e > 0$ can be chosen arbitrarily small,
it then follows by Lemma \ref{rel compact subsets}
that $K_{u; -s}(\Pi (A))$ is relatively compact in
$h^{-s}(\N_0)$. Since by definition
$$
\big( K_{u; -s}(\Pi u) \big)_n = 
\langle \Pi u | f_n (\cdot, u) \rangle
\, , \quad \forall n \ge 0\, ,
$$ 
and since by \eqref{formula coeff of u},
$$
\zeta_n(u) \simeq \frac{1}{\sqrt{n}}
\langle \Pi u | f_n (\cdot, u) \rangle   \quad \text{as }\,\,  n \to \infty
$$
uniformly with respect to $u \in A,$ it follows that $\Phi(A)$
is relatively compact in $h^{1/2 -s}_+$.\\
Now let us turn to $\Phi^{-1}$. 
By Proposition \ref{extension Phi, part 2}(ii),
$\Phi^{-1}: h^{1/2-s}_+ \to H^{-s}_{r,0}$
is sequentially weakly continuous. 
To show that this map is continuous
it then suffices to prove that the image $\Phi^{-1}(B)$ of any
relatively compact subset $B$ of $h^{1/2-s}_+$
 is relatively compact in $ H^{-s}_{r,0}$.
By the same arguments as above one sees that 
$\Phi^{-1}: h^{1/2-s}_+ \to H^{-s}_{r,0}$ 
is also continuous.
\hfill $\square$

\smallskip

It remains to state Lemma \ref{rel compact subsets},
used in the proof of  
Proposition \ref{extension Phi, part 2}$(iii)$.
It concerns the well known characterization
of relatively compact subsets of $H^{-s}_{r,0}$
in terms of the Fourier expansion
$u(x) = \sum_{n \ne 0} \widehat u(n) e^{inx}$
of an element $u$ in $H^{-s}_{r,0}$.
\begin{lemma}\label{rel compact subsets}
Let $0 < s < 1/2$ and $A \subset H^{-s}_+$.
Then $A$ is relatively compact in $H^{-s}_+$
if and only if for any $\e > 0,$ there exist
$N_\e \ge 1$ and $R_\e > 0$  so that for any $f \in A,$ the sequence $\xi_n:=\hat f(n), n\ge 0$ satisfies
$$
\big(\sum_{n > N_\e} |n|^{-2s} |\xi_n|^2\big)^{1/2} 
< \e
\, , \qquad 
\big( \sum_{0 <n \le N_\e} |\xi_n|^2 \big)^{1/2} 
< R_\e \, .
$$
The latter conditions on $(\xi_n)_{n \ge 0}$ characterize relatively compact subsets of $h^{-s}(\N_0)$.
\end{lemma}

\smallskip
\noindent
{\em Proof of Theorem \ref{extension Phi}.}
The claimed statements follow from
Proposition \ref{extension Phi, part 1} and
Proposition \ref{extension Phi, part 2}.
In particular, item $(ii)$ of Theorem \ref{extension Phi} 
follows from Proposition \ref{extension Phi, part 2}$(i)$ by setting
$$F_s(R):=\sup_{\| \xi \|_{\frac 12-s}\leq R}\| \Phi ^{-1}(\xi)\|_{-s}\ .$$
\hfill $\square$
 

\section{Solution maps $\mathcal S_0$, $\mathcal S_B$ and $\mathcal S_c$, $\mathcal S_{c,B}$} \label{mathcal S_B}

In this section we provide results related to the solution
map of \eqref{BO}, which will be used to prove 
Theorem \ref{Theorem 1} in the subsequent section.

\medskip

\noindent
{\em Solution map $\mathcal S_{B}$ and its extension. }
First we study the map $\mathcal S_B,$ defined in Section 2
on $h^{1/2}_+$. Recall that by 
\eqref{frequencies in Birkhoff} -- \eqref{frequency map},
the nth frequency of \eqref{BO} is a real valued map defined on $\ell^2_+$ by
$$
\omega_n(\zeta) := n^2 -2\sum_{k=1}^{n} k |\zeta_k|^2
- 2n\sum_{k=n+1}^{\infty} |\zeta_k|^2\, .
$$
For any $0 < s \le 1/2$, the map $\mathcal S_B$ naturally extends to $h^{1/2 -s}_+$, mapping initial data 
$\zeta(0) \in h^{1/2 -s}_+$ to the curve 
\begin{equation}\label{formula for S_B}
\mathcal S_B(\cdot, \zeta(0)): \R \to h^{1/2 -s}_+\ , \  t \mapsto 
\mathcal S_B(t, \zeta(0)):=
\big( \zeta_n(0)e^{it\omega_n(\zeta)} \big)_{n \ge 1} \, .
\end{equation}
We first record the following properties of the frequencies.
\begin{lemma}\label{continuity frequency map} $(i)$ For any $n \ge 1$,
 $\omega_n: \ell^2_+ \to \R$ is continuous and
$$
| \omega_n(\zeta) - n^2 | \le 2 n \|\zeta\|_0^2 \, , 
 \,\,\, \forall \zeta \in  \ell^2_+\, ;
 \quad 
 | \omega_n(\zeta) - n^2 | \le  2\|\zeta\|_{1/2}^2 
\, , \,\,\, \forall \zeta \in  h^{1/2}_+\, .
$$
$(ii)$ For any $0 \le s < 1/2$, \, 
 $\omega_n: h^{1/2 - s}_+ \to \R$
is sequentially weakly continuous.
\end{lemma}
\begin{proof}
Item $(i)$ follows in a straightforward way from the formula
\eqref{frequencies in Birkhoff} of $\omega_n$. 
Since for any $0 \le s < 1/2$, $h^{1/2 - s}_+$ 
compactly embeds into $\ell^2_+$, item $(ii)$ follows from $(i)$.
\end{proof}
From Lemma \ref{continuity frequency map} one infers  the following properties of $\mathcal S_B$. We leave the easy proof to the reader.
\begin{proposition}\label{S in Birkhoff continuous}
For any $0 \le s \le 1/2,$ the following holds:\\
$(i)$ For any  initial data $\zeta(0) \in h^{1/2-s}_+$, 
$$
\R \to h^{1/2-s}_+\, , \, t \mapsto 
\mathcal S_B(t, \zeta(0))
$$ 
is continuous.\\
$(ii)$ For any $T > 0,$
$$
\mathcal S_B: h^{1/2-s}_+ \to C([-T, T], h^{1/2-s}_+), \, 
\zeta(0) \mapsto  \mathcal S_B(\cdot, \zeta(0))\, ,
$$
is continuous and for any $t \in \R$,
$$
\mathcal S^t_B : h^{1/2-s}_+ \to h^{1/2-s}_+, \,
\zeta(0) \mapsto  \mathcal S_B(t, \zeta(0)) \, ,
$$
is a homeomorphism.
\end{proposition} 

\smallskip

\noindent
{\em Solution map $\mathcal S_0$ and its extension.}
Recall that in Section \ref{Birkhoff map} we introduced the
solution map $\mathcal S_0$ of \eqref{BO}
on the subspace space $L^2_{r,0}$ of $L^2_r$,
consisting of elements in $L^2_r$ with average $0$,
in terms of the Birkhoff map $\Phi$,
\begin{equation}\label{flow map BO}
\mathcal S_0 = \Phi^{-1} \mathcal S_B \Phi\, : \,  
L^2_{r,0} \to C(\R, L^2_{r,0}) \, .
\end{equation}
Theorem \ref{extension Phi} will now be applied to
prove the following result about the extension
of $\mathcal S_0$ to the Sobolev space $H^{-s}_{r,0}$
with $0 < s < 1/2$, consisting of elements in
$H^{-s}_r$ with average zero.
It will be used in Section \ref{Proofs of main results} 
to prove Theorem \ref{Theorem 1}.

\begin{proposition}\label{solution map S_0 for neg}
For any $0 \le s < 1/2$, the following holds:\\
$(i)$ The Benjamin-Ono equation is globally $C^0-$well-posed on $H^{-s}_{r,0}$.\\
$(ii)$ There exists an increasing function
$F_s :\R_{\ge 0} \to \R_{\ge 0} $ so that
$$
\| u \|_{-s} \le F_s\big(\|\Phi(u)\|_{1/2-s}\big)\, , \quad \forall u \in H^{-s}_{r,0} \, .
$$
In particular, 
for any initial data $u(0) \in H^{-s}_{r,0}$,
\begin{equation}\label{a priori bound}
\sup_{t\in \R}\| \mathcal S^t_0(u(0)) \|_{-s}\leq 
F_s\big(\|\Phi(u(0))\|_{1/2-s}\big)\, .
\end{equation}
\end{proposition}
\begin{remark}
$(i)$ By the trace formula \eqref{trace formula},
for any $u(0) \in L^2_{r,0}$, 
estimate \eqref{a priori bound} can be improved as follows,
$$
\|u(t) \| = \sqrt{2} \|\Phi(u(0))\|_{1/2} =
\|u(0)\|\, , \quad \forall t \in \R\, .
$$
$(ii)$ We refer to the comments of Theorem \ref{Theorem 1}
in Section \ref{Introduction} for a discussion of the
recent results of Talbut \cite{Tal}, related to \eqref{a priori bound}.
\end{remark}
\begin{proof} Statement $(i)$  follows from the corresponding statements for $S_B$ in Proposition \ref{S in Birkhoff continuous}
and the continuity properties of $\Phi $ and $\Phi^{-1}$ stated in Theorem \ref{extension Phi}. \\
$(ii)$ By Theorem \ref{extension Phi} there exists an increasing
function $F_s : \R_{\ge 0} \to \R_{\ge 0}$ so that for any
$u \in H^{-s}_{r,0}$, 
$\| u \|_{-s} \le F_s(\| \Phi(u) \|_{1/2 -s})$.
Since the norm of $h^{1/2 -s}$
is left invariant by the flow $\mathcal S_B^t$ , 
it follows that 
for any initial data $u(0) \in H^{-s}_{r,0}$, one has
$
\sup_{t\in \R}\| \mathcal S^t(u(0))\|_{-s}\leq 
F_s\big(\|\Phi(u(0))\|_{1/2-s}\big)\, .
$
\end{proof}

\smallskip

\noindent
{\em Solution map $\mathcal S_{c}$.}\label{solution map Sc}
Next we introduce the solution map $\mathcal S_c$
where $c$ is a real parameter.
Let $v(t,x)$ be a solution of \eqref{BO} with initial data $v(0) \in H^s_r$ and $s > 3/2$,
satisfying the properties $(S1)$ and $(S2)$ stated in
Section \ref{Introduction}. 
By the uniqueness property in $(S1)$, it then follows that
 \begin{equation}\label{identity of solutions}
 v(t, x) = u(t, x - 2ct) + c\, , \quad  c= \langle v(0) | 1 \rangle
 \end{equation}
 where $u \in C(\R, H^s_{r,0}) \cap C^1(\R, H^{s-2}_{r,0})$
 is the solution of the initial value problem
 \begin{equation}\label{BO with mean zero}
\partial_t u = H\partial^2_x u - \partial_x (u^2)\,, \qquad  u(0) = v(0) - \langle v(0) | 1 \rangle \, ,
\end{equation}
satisfying $(S1)$ and $(S2)$. 
It then follows that
$ w(t,x) := u(t, x -2ct)$ satisfies $w(0) = u(0)$ and
\begin{equation}\label{BO with c}
\partial_t w = H\partial^2_x w - \partial_x (w^2)  + 2c \partial_x w \, .
\end{equation}
By \eqref{identity of solutions},
the solution map of \eqref{BO with c}, denoted by $\mathcal S_c$, 
is related to the solution map $\mathcal S$ of \eqref{BO}
(cf. property $(S2)$ stated in Section \ref{Introduction}) by
\begin{equation}\label{formula with c}
\mathcal S(t, v(0)) = 
\mathcal S_{[v(0)]} \big(t, v(0) - [v(0)] \big) + [v(0)]
\ , \quad [v(0)] := \langle v(0) | 1 \rangle \, .
\end{equation}
In particular, for any $s > 3/2$,
\begin{equation}\label{solution map for BO with c}
\mathcal S_c: H^s_{r,0} \to  C(\R, H^s_{r, 0}), \ w(0) \mapsto \mathcal S_c(\cdot, w(0) ) 
\end{equation}
is well defined and continuous.
Molinet's results in \cite{Mol} ( cf. also \cite{MP}) imply that the solution map $\mathcal S_c$ continuously 
extends to any
Sobolev space $H^s_{r,0}$ with $0 \le s \le 3/2$. More precisely, for any such $s$,
$\mathcal S_c: H^s_{r,0} \to 
C(\R, H^s_{r,0})$ is continuous
and for any $v_0 \in H^s_{r,0}$, $ \mathcal S_c(t, w_0)$ satisfies equation \eqref{BO}  
in $H^{s-2}_r$.

\medskip

\noindent
{\em Solution map $\mathcal S_{c, B}$ and its extension.}
Arguing as in Section \ref{Birkhoff map}, we use
Theorem \ref{main result}, to express
the solution map $\mathcal S_{c, B}$,
corresponding to the equation \eqref{BO with c}
in Birkhoff coordinates.
Note that \eqref{BO with c} is Hamiltonian, $\partial_t w = \partial_x \nabla \mathcal H_c$, with Hamiltonian
$$
\mathcal H_c : H^s_{r,0} \to \R \, , \qquad  \mathcal H_c (w) =  \mathcal H(w)  + 2c \mathcal H^{(0)}(w)
$$
where by \eqref{CL},  $\mathcal H^{(0)}(w)  = \frac{1}{2\pi} \int_0^{2\pi} \frac{1}{2} w^2 dx$.
Since by Parseval's formula, derived  in \cite[Proposition 3.1]{GK}, 
$ \frac{1}{2\pi} \int_0^{2\pi} \frac{1}{2} w^2 dx = \sum_{n=1}^\infty n | \zeta_n |^2$
one has
$$
\mathcal H_{c, B} (\zeta) := \mathcal H_c (\Phi^{-1} (\zeta)) = \mathcal H_{B} (\zeta) +2c  \sum_{n=1}^\infty n | \zeta_n |^2 \, ,
$$
implying that the corresponding frequencies $\omega_{c, n} $, $n \ge 1$, are given by
\begin{equation}\label{frequ for c ne 0}
\omega_{c, n} (\zeta) =   \partial_{|\zeta_n|^2} \mathcal H_{c, B} (\zeta)= \omega_{n} (\zeta) + 2c n \, .
\end{equation}
For any $c \in \R$, denote by $\mathcal S_{c, B}$
the solution map of \eqref{BO with c}
when expressed in Birkhoff coordinates,
\begin{equation}\label{evolution S_{c, B}}
\mathcal S_{c, B} : h^{1/2}_+ \to C(\R,  h^{1/2}_+ )\, , \zeta(0)  \mapsto  [ t \mapsto  \big(  \zeta_n(0) e^{i t \omega_{c,n}(\zeta(0))}\big)_{n \ge 1} ] \,.
\end{equation}
Note that $\omega_{0,n} = \omega_n$ and hence
$\mathcal S_{0, B}  =  \mathcal S_{ B}$. 
Using the same arguments as in the proof of 
Proposition \ref{S in Birkhoff continuous} one obtains the following
\begin{corollary}\label{solution map S_{c,B} for neg}
The statements of Proposition \ref{S in Birkhoff continuous}
continue to hold for $\mathcal S_{c, B}$ with $c \in \R$
arbitrary.
\end{corollary}

\smallskip

\noindent 
{\em Extension of the solution map $\mathcal S_{c}$.}
Above, we introduced the
solution map $\mathcal S_c$
on the subspace space $L^2_{r,0}$. One infers from
\eqref{formula with c} that
\begin{equation}\label{flow map BO}
\mathcal S_c = \Phi^{-1} \mathcal S_{c,B} \Phi\, : \,  
L^2_{r,0} \to C(\R, L^2_{r,0}) \, .
\end{equation}
Using the same arguments as in the proof of 
Proposition \ref{solution map S_0 for neg} 
one infers from Corollary \ref{solution map S_{c,B} for neg} 
the following results, concerning the extension
of $\mathcal S_c$ to the Sobolev space $H^{-s}_{r,0}$
with $0 < s < 1/2$.
\begin{corollary}\label{solution map S_c for neg}
The statements of Proposition \ref{solution map S_0 for neg}
continue to hold for $\mathcal S_{c, B}$ with $c \in \R$
arbitrary.
\end{corollary}


\section[Proofs]{Proofs Theorem \ref{Theorem 1}, Theorem \ref{Theorem 3}, and Theorem \ref{Theorem 4}} \label{Proofs of main results}

\noindent
{\em Proof of Theorem \ref{Theorem 1}.}
Theorem \ref{Theorem 1} is a straightforward consequence
of Proposition \ref{solution map S_0 for neg} and 
Corollary \ref{solution map S_c for neg}.
\hfill $\square$

\medskip
\noindent
{\em Proof of Theorem \ref{Theorem 3}.}
We argue similarly as in the proof of \cite[Theorem 2]{GK}.
Since the case $c \ne 0$ is proved be the same 
arguments we only consider the case $c=0$.
Let $u_0 \in H^{-s}_{r,0}$ with $0\le s < 1/2$
and let $u(t):= \mathcal S_0(t, u_0)$. By formula  
\eqref{formula for S_B}, 
$\zeta(t) := \mathcal S_B(t, \Phi(u_0))$ evolves
on the torus ${\rm{Tor}}(\Phi(u_0))$, defined by \eqref{def tor}.\\
$(i)$ Since ${\rm{Tor}}(\Phi(u_0))$ 
is compact in $h^{1/2 -s}_+$
and $\Phi^{-1}: h^{1/2 -s}_+ \to H^{-s}_{r,0}$
is continuous, $\{ u(t) \ : \ t \in \R \}$
is relatively compact in $H^{-s}_{r,0}$.\\
$(ii)$
In order to prove that $ t \mapsto u(t)$ is almost periodic, we appeal to Bochner's  characterization of such functions (cf. e.g. \cite{LZ}) : a bounded continuous function $f:\R \to X$
with values in a Banach space $X$ is almost periodic if and only if the set $\{ f_\tau, \tau \in \R\}$  of functions defined by 
$f_\tau (t):=f(t+\tau)$
is relatively compact in the space $\mathcal C_b(\R, X)$ of bounded continuous functions on $\R$ with values in $X$.
Since $\Phi: H^{-s}_{r,0} \to h^{1/2 -s}_+ $ 
is a homeomorphism, in the case at hand, it suffices to prove that for every sequence $(\tau_k)_{k\ge 1}$ of real numbers,  the sequence 
$f_{\tau_k}(t):=\Phi (u(t + \tau_k))$, $k\ge 1,$ 
in $C_b(\R,h^{1/2 -s}_+)$ admits a subsequence which converges uniformly in
$C_b(\R,h^{1/2 - s}_+)$. Notice that
$$f_{\tau_k}(t)=\big(\zeta_n(u(0)){\rm e}^{i\omega_n(t+\tau_k)}\big)_{n\ge 1}.$$
 By Cantor's diagonal process
and since the circle is compact, 
there exists a subsequence of $(\tau_k)_{k\ge 1}$,
again denoted by $(\tau_k)_{k\ge 1}$, so that for any $n \ge 1,$ $\lim_{k \to \infty}{\rm e}^{i\omega_n\tau_k}$ exists,
implying that the sequence of functions $f_{\tau_k}$
converges 
uniformly in $C_b(\R,h^{1/2 -s}_+)$.
\hfill $\square$

\medskip

\noindent
{\em Proof of Theorem \ref{Theorem 4}.}
Since the general case can be proved by the same arguments
we consider only the case $c = 0$.
By  \cite[Proposition B.1]{GK}, the traveling wave solutions
of the BO equation on $\T$ coincide with the one gap solutions.
Without further reference, we use notations and results from \cite[Appendix B]{GK}, where one gap potentials have
been analyzed. 
Let $u_0$ be an arbitrary one gap potential. Then
$u_0$ is $C^{\infty}-$smooth and there
exists $N \ge 1$ so that
$\gamma_N(u_0) > 0$ and $\gamma_n(u_0) = 0$ for any
$n \ne N$. 
Furthermore, the orbit of the corresponding one gap solution is given by $\{ u_0(\cdot + \tau) \, : \, \tau \in \R \}$.
Let $0 \le s < 1/2$. 
It is to prove that for any $\e > 0$ there exists $\delta > 0$
so that for any $v(0) \in H^{-s}_{r,0}$ with 
$\|v(0) - u_0 \|_{-s} < \delta$ one has
\begin{equation}\label{epsilon orbit close}
\sup_{t \in \R} \inf_{\tau \in \R}
\| v(t) - u_0( \cdot + \tau) \|_{-s} < \e \, .
\end{equation}
To prove the latter statement, we argue by contradiction.
Assume that there exists $\e > 0$, a sequence 
$(v^{(k)}(0))_{k \ge 1}$ in $H^{-s}_{r,0}$, and
a sequence $(t_k)_{k \ge 1}$  so that
$$
\inf_{\tau \in \R} 
\| v^{(k)}(t_k) - u_0( \cdot + \tau) \|_{-s} \ge \e \, , 
\quad \forall k \ge 1 \, , \qquad
\lim_{k \to \infty } \| v^{(k)}(0) - u_0 \|_{-s} = 0 \, .
$$
Since $A:= \{v^{(k)}(0) \, | \, k \ge 1 \} \cup \{ u_0 \}$
is compact in $H^{-s}_{r,0}$ and $\Phi$ is continuous, 
$\Phi(A)$ is compact in $h^{1/2-s}_+$ and
$$
\lim_{k \to \infty} 
\| \Phi(v^{(k)}(0)) - \Phi(u_0) \|_{1/2 -s} = 0 \, .
$$
It means that
$$
\lim_{k \to \infty} \sum_{n = 1}^{\infty}
n^{1 - 2s} | \zeta_n(v^{(k)}(0)) - \zeta_n(u_0) |^2 = 0 \, .
$$
Note that for any $k \ge 1$,
$$
\zeta_n(v^{(k)}(t_k)) = 
\zeta_n(v^{(k)}(0)) \ e^{i t_k \omega_n(v^{(k)}(0))}\ ,
\quad \forall n \ge 1
$$
and $\zeta_n(u(t_k)) = \zeta_n(u_0) = 0$ 
for any $n \ne N$. Hence
\begin{equation}\label{estimate for n ne N}
\lim_{k \to \infty} \sum_{n \ne  N}
n^{1 - 2s} | \zeta_n(v^{(k)}(t_k)) |^2 = 0\ , 
\end{equation}
and since 
$| \zeta_N(v^{(k)}(t_k))| = | \zeta_N(v^{(k)}(0))|$ 
one has
\begin{equation}\label{estimate for n equal N}
\lim_{k \to \infty} \big| | \zeta_N(v^{(k)}(t_k))|
-  | \zeta_N(u_0) | \big| = 0\ ,
\end{equation}
implying that 
$\sup_{k \ge 1} | \zeta_N(v^{(k)}(t_k))| < \infty$.
It thus follows that the subset $\{ \Phi(v^{(k)}(t_k))  : k \ge 1 \}$
is relatively compact in $h^{1/2 -s}$ and hence 
$\{ v^{(k)}(t_k) :  k \ge 1 \}$ relatively compact 
in $H^{-s}_{r,0}$. Choose a subsequence 
$( v^{(k_{j})}(t_{k_{j}}))_{j \ge 1}$ which
converges in $H^{-s}_{r,0}$ and denote its limit
by $w \in H^{-s}_{r,0}$. By 
\eqref{estimate for n ne N}--\eqref{estimate for n equal N}
one infers that there exists $\theta \in \R$ so that 
$$
\zeta_n(w) =0\ , \quad \forall n \ne N\, , \qquad
\zeta_N(w) = \zeta_{N}(u_0)e^{i \theta}\ .
$$
As a consequence, $w(x) = u_0(x+ \theta / N)$, contradicting
the assumption that 
$\inf_{\tau \in \R} 
\| v^{(k)}(t_k) - u_0( \cdot + \tau) \|_{-s} \ge \e$
for any $k \ge 1$.
\hfill $\square$

 
 \section[Proof of Theorem 2]{Proof of Theorem \ref{Theorem 2}}\label{Proof of Theorem 2}
 
 In this section we prove Theorem \ref{Theorem 2}. 
First we need to to make some preparations. We consider potentials of the form
$u(x)=v( {\rm e}^{ix} )+\overline {v( {\rm e}^{ix} )}$, where $v$ is a Hardy function, defined in the unit disc by
$$
v(z)= \frac{\e qz}{1-qz}\ , \quad 0< \e < q<1 \ ,  \qquad  |z|<1\ .
$$
Note that
\begin{equation}\label{normofu}
\| u\|_{-1/2}^2=2\e ^2\sum_{n=1}^\infty n^{-1}q^{2n}=-2\e ^2\log (1-q^2)\ .
\end{equation}
We want to investigate properties of the Birkhoff coordinates of $u$. To this end we consider the Lax operator $L_u= D -T_u$
for $u$ of the above form. Since for any $f\in H^1_+$ and $z \in \C$ in the unit disc
$$
T_uf (z)= \Pi \big((v + \overline v) f \big)(z) = v(z)f(z)+\e q \frac{f(z)-f(q)}{z-q} \ ,
$$
the eigenvalue equation $L_uf - \lambda f = 0$, $\lambda \in \R$, reads
\begin{equation}\label{eigeneq}
zf'(z)-\left ( \frac{\e qz}{1-qz}+\frac{\e q}{z-q}-\mu \right )f(z)=f(q)\frac{\e q}{q-z}\ ,\qquad  |z|<1\ , 
\end{equation}
where we have set $\mu:= - \lambda$. Note that if $-\mu$ is an eigenvalue, then 
the eigenfunction $f(z)$ is holomorphic in $|z| < 1$. Evaluating \eqref{eigeneq} 
in such a case at $z=0$, one infers
\begin{equation}\label{f(q)}
f(q) = \frac{\e + \mu}{\e} f(0) \ .
\end{equation}
Define for $|z|<q^{-1}$, $z\notin [0,q^{-1})$,
$$
\psi(z) :=\frac{z^{\e +\mu} (1-qz)^\e }{(q-z)^\e }\ ,
$$
with the branches of the fractional powers chosen as  
$$
\arg (z)\in (0,2\pi)\ , \quad \arg (q-z)\in (-\pi ,\pi) \ , \quad \arg (1-q z)\in \left (-\frac \pi 2, \frac \pi 2\right ) \ .
$$ 
Then \eqref{eigeneq} reads
$$
\frac{d}{dz}(\psi (z)f(z))=f(q)\frac{\e q \psi (z)}{z(q-z)}\ .
$$
As a consequence, $f(q)\ne 0$ since otherwise $\psi (z)f(z)$ would be constant and hence $\psi (z)^{-1}$  holomorphic on the whole unit disc, which is impossible since
\begin{equation}\label{compatibility}
\frac{\psi (t+i0)}{\psi (t-i0)}=
\frac{(t+i0)^{\e +\mu}(q-t+i0)^\e }{(q-t-i0)^\e (t-i0)^{\e +\mu}}
=\begin{cases} {\rm e}^{-2\pi i(\e +\mu)}, \ t \in  (0, q) \\
 {\rm e}^{-2\pi i\mu},  \quad  t \in (q, q^{-1})  \end{cases} \ 
\end{equation}
and ${\rm e}^{2\pi i\e} \ne 1$ for any $0 < \e < 1$.
For what follows it is convenient to normalize the eigefunction $f$ by $f(q)=1$. Then the eigenvalue equation reads
\begin{equation}\label{eigenvalue equation}
\frac{d}{dz}(\psi (z)f(z))= g(z) \ , \qquad  g(z):= \frac{\e q \psi (z)}{z(q-z)}, \qquad |z| < 1 \ .
\end{equation}
Our goal is to prove that for an appropriate choice of the parameters $\e$ and $q$,  $\mu = - \lambda_0 (u) $ becomes
arbitrarily large. So from now on we assume that $\mu > 0$. Note that by \eqref{eigenvalue equation} one has
$$
f(z)=\frac{1}{\psi (z)}\int_0^zg(\zeta )\, d\zeta \ , \qquad \{ |z| < q^{-1} \}\setminus [0, q^{-1}) \ .
$$
Hence $-\mu$ will be an eigenvalue of $L_u,$ 
if the right hand side of the latter expression extends to a holomorphic function
in the disc $\{ |z| < q^{-1} \}$. It means that $f$ can can be continuously extended to $z=0$ and $z=q$ and that 
\begin{equation}\label{compatibility on (0, 1/q)}
f(t+i0)=f(t-i0) \ , \qquad \forall t\in (0,q)\cup (q, q^{-1}) \ .
\end{equation}
It is straightforward to check that in the case $\mu > 0$, $f$ extends continuously to $z=0$.
Furthermore, the identity \eqref{compatibility on (0, 1/q)} is verified for any $t\in (0,q)$ since
$$\frac{g(s+i0)}{g(s-i0)}={\rm e}^{-2\pi i(\e +\mu)}=\frac{\psi (t+i0)}{\psi (t-i0)}\ ,\qquad \forall 0<s<t\ .$$
It is then straightforward to verify that $f$ extends continuously to $z=q$ and that $f(q)=1$.
Next, let us examine the condition $f(t+i0)=f(t-i0)$ for $t\in (q,1)$. Notice that 
$$g(z)=\frac{d}{dz}\Big( \frac{z^\e}{(q-z)^\e } \Big)z^\mu (1-qz)^\e  \ ,$$
so that with $h(z):=\frac{z^\e}{(q-z)^\e } \frac{d}{dz} (z^\mu (1-qz)^\e )$ one has, integrating by parts,
$$
\int_0^z g(\zeta )\, d\zeta 
=\frac{z^\e}{(q-z)^\e }z^\mu (1-qz)^\e  - \int_0^z\ h(\zeta )\, d\zeta  
= \psi(z) - \int_0^z\ h(\zeta )\, d\zeta  \ .
$$
As a consequence, the condition $f(t+i0)=f(t-i0)$, $t\in (q,q^{-1})$, reads
$$
\int_{0}^{t}\ h(\zeta +i0 )\, d\zeta =\frac{\psi (t+i0)}{\psi (t-i0)}\int_{0}^{t}\ h(\zeta -i0 )\, d\zeta \ , \qquad \forall q < t < q^{-1} \ .
$$
Since by \eqref{compatibility},
$\psi(\zeta +i0 )= {\rm e}^{-2 \pi i \mu} \psi(\zeta -i0)$ 
and $h(\zeta +i0 )= {\rm e}^{-2 \pi i \mu} h(\zeta -i0)$ 
for any $q< \zeta <  q^{-1}$, 
we conclude that the following condition 
$$
\int_0^q\ h(\zeta +i0 )\, d\zeta = {\rm e}^{-2 \pi i \mu}  \int_{0}^{q}\ h(\zeta -i0 )\, d\zeta  
$$
is necessary and sufficient for $- \mu < 0 $ to be an eigenvalue of $L_u$.
After simplification and using again that ${\rm e}^{2i\pi \e} \ne 1$, this condition reads
\begin{equation}\label{F0}
F(\mu ,\e ,q) := \int_0^q \frac{t^\e}{(q-t)^\e } \frac{d}{dt} (t^\mu (1-qt)^\e )\, dt =0
\end{equation}
or
\begin{equation}\label{F1}
F(\mu ,\e ,q) =\int_0^q \frac{t^{\e +\mu}(1-qt)^\e }{(q-t)^\e } \left (\frac \mu t -\frac{\e q}{1-qt}\right )\, dt =0\ .
\end{equation}
Notice that  if $\mu \geq \frac{\e q^2}{1-q^2}$, then
the latter integrand is strictly positive for any $0 < t  < q$. In particular, one has
\begin{equation}\label{Fpositive}
F \Big(  \frac{\e q^2}{1-q^2}  ,\e ,q \Big) > 0\ .
\end{equation}
On the other hand, let us fix $\mu >0$ and study the limit of $F(\mu ,\e ,q)$ as $(\e, q) \to (0, 1)$. 
Clearly, one has 
\begin{equation}\label{F2}
\lim_{(\e, q) \to (0, 1) }\int_0^q\frac{t^{\e +\mu}(1-qt)^\e }{(q-t)^\e } \frac \mu t \, dt 
=  \int_0^1 \mu t^{\mu -1}\, dt =1\ .
\end{equation}
To compute the limit of the remaining part of $F(\mu ,\e ,q)$ is more involved. 
For any given fixed positive parameter $1 - q^2 < \theta < 1$, split the integral
$$I(\mu ,\e ,q):=\int_0^q \frac{t^{\e +\mu}(1-qt)^\e}{(q-t)^\e}\, 
\frac{\e q}{1-qt}\, dt $$
into three parts,
$$
I_1(\mu ,\e ,q; \theta) +I_2(\mu ,\e ,q; \theta) + I_3(\mu ,\e ,q)\ := \int _0^{q(1-\theta)}+\int_{q(1-\theta )}^{q^3}+\int_{q^3}^q \ .
$$
It is easy to check that
$$0\leq I_1(\mu ,\e ,q; \theta)  \leq  C_1(\theta) \e $$
and that with the change of variable $t := q- q(1-q^2)y$ in $I_3$, one has
$$
0 \le I_3(\mu ,\e ,q) 
=\e q^{\mu +2}\int_0^1(1-(1-q^2))^{\e +\mu}(1+q^2y)^\e y^{-\e }\frac{dy}{1+q^2y} \le C_3 \e \ .
$$
Using the same change of variable in $I_2$, we obtain
$$
I_2(\mu ,\e ,q; \theta) 
=\e q^{\mu +2}\int_1^{\theta /(1-q^2)}(1-(1-q^2)y)^{\e +\mu} \frac{dy}{y^{\e }(1+q^2y)^{1-\e }} \ .
$$
Note that for $1 \le y \le \theta /(1-q^2)$, one has
$1-\theta \leq 1-(1-q^2)y\leq 1$, and hence 
$$
(1-\theta)^{\mu+1} \leq (1-(1-q^2)y)^{\mu + \e} \leq 1 \ .
$$ 
Since
$$
\frac{1+y}{2} \leq y \leq  1+ y\ ,\qquad  q^2(1+y) \leq 1+q^2y \leq 1+y ,
$$
we then infer that 
$$
\e q^{\mu +2}(1-\theta)^{\mu+1} \int_1^{\theta /(1-q^2)}\frac{dy}{1+y} \ \leq \ I_2 \ \leq  \ 
\e q^{2\e +\mu }2^\e \int_1^{\theta /(1-q^2)}\frac{dy}{1+y}\ .
$$
Using that as $q \to 1,$
$$
 \frac{1}{ - \log(1 - q^2)} \int_1^{\theta /(1-q^2)}\frac{dy}{1+y} =  1 - \frac{\log((1-q^2 + \theta)/2)}{\log(1-q^2)}  \to 1
$$
we then obtain
$$
(1-\theta )^{\mu+1} \leq 
\liminf_{(\e, q) \to (0, 1)}  \frac{I_2(\mu ,\e ,q; \theta) }{-\e \log (1-q^2)} 
\le 
\limsup_{(\e, q) \to (0, 1)} \frac{I_2(\mu ,\e ,q; \theta) }{-\e \log (1-q^2)}\leq 1\ .
$$
Summarizing,  we have proved that for any $0 < \theta < 1$,
$$
(1-\theta )^{\mu+1} \leq \liminf_{(\e, q) \to (0, 1)} \frac{I(\mu, \e ,q)}{-\e \log (1-q^2)} \ \le \ 
\limsup_{(\e, q) \to (0, 1)} \frac{I(\mu, \e ,q)}{-\e \log (1-q^2)}\leq 1\ .
$$
Letting $\theta \to 0$, we conclude that
\begin{equation}\label{F3}
I(\mu, \e ,q)= -\e \log(1-q^2) (1+o(1))\to +\infty 
\end{equation}
for $(\e, q)$ satisfying
\begin{equation}\label{Ftoinfty}
\e \log (1-q^2)\to -\infty \ .
\end{equation}
Therefore, if  \eqref{Ftoinfty} holds, then
by \eqref{F1},  \eqref{F2},  and \eqref{F3} 
$$ 
F(\mu ,\e ,q)\to -\infty \ , \qquad \forall \mu >0\ .
$$
For any $k \ge 1,$ let
$$
q_k^2 :=1 - {\rm e}^{-\e_k^{-3/2}} 
$$
with $0 < \e_k < q_k$ so small that
\begin{equation}\label{Fnegative}
F(k, \e_k, q_k) <0 , \qquad 
\frac{\e_k q_k^2}{1-q_k^2}=\e _k{\rm e}^{\e_k^{-3/2}}(1-{\rm e}^{-\e_k^{-3/2}}) >k \ .
\end{equation}
Set 
\begin{equation}\label{definition u^{(k)}}
u^{(k)}(x) := 2{\rm Re}\big(\frac{\e_k q_k{\rm e}^{ix}}{1-q_k{\rm e}^{ix}}\big)\ .
\end{equation}
\begin{lemma}\label{ill1}
For any $k \ge 1$, $u^{(k)} \in \bigcap_{n \ge 1}H^{n}_{r,0}$, and 
$\lim_{k \to \infty} \| u^{(k)} \|_{-1/2} = 0$ as well as $\lim_{k \to \infty} \lambda_0(u^{(k)}) = - \infty$.
\end{lemma}
\begin{proof}
Expanding $u^{(k)}$, $k \ge 1$, in its Fourier series, it is straightforward to check that $u^{(k)} \in \bigcap_{n \ge 1}H^{n}_{r,0}$.
By \eqref{normofu} we have
$\| u^{(k)}\|_{H^{-1/2}}^2=\sqrt{\e _k} \to 0$ and by \eqref{Fpositive}, \eqref{Fnegative}, 
$\lambda_0(u^{(k)})<-k$ and hence $\lim_{k \to \infty} \lambda_0(u^{(k)}) = -\infty $.
\end{proof}
In a next step we prove that for $u$ of the form $u = 2{\rm Re}\big(\frac{\e q{\rm e}^{ix}}{1-q{\rm e}^{ix}}\big)$,
$L_u$ has only one negative eigenvalue. More precisely the following holds.
\begin{lemma}\label{ill2} 
For any $0 < \e < q < 1,$ $F(\cdot, \e, q)$ has precisely one zero in $\R_{>0}$.
It means that $\lambda_0(u)$ is the only negative eigenvalue of $L_u$ and thus
$\lambda_1(u) \ge 0$. Furthermore,  $\lambda_1(u) = 1 - \sum_{k \ge 2} \gamma_k(u) \le 1$ and 
\begin{equation}\label{estimgamma}
0 \le \sum_{n=2}^\infty \gamma_n(u)\leq 1 \quad  \text{ and } \quad \gamma_n(u) > 0 \ , \quad \forall n \ge 2.
\end{equation}
\end{lemma}
\begin{proof}
The proof relies on an alternative formula for $F(\mu, \e ,q)$, obtained from \eqref{F0} by integrating by parts.
Choosing $q^\mu (1-q^2)^\e -t^\mu (1-qt)^\e$ as antiderivative of $ \frac{d}{dt} (t^\mu (1-qt)^\e )$ one gets
\begin{equation}\label{alterF}
F(\mu, \e ,q)=\e q\int_0^q \frac{t^\e}{(q-t)^\e } \frac{ q^\mu (1-q^2)^\e -t^\mu (1-qt)^e }{t(q-t)}\, dt \ .
\end{equation}
Consequently,
$$\partial_\mu F (\mu, \e ,q)=\e q\int_0^q \frac{t^\e}{(q-t)^\e } \frac{ q^\mu (1-q^2)^\e \log q -t^\mu (1-qt)^\e \log t }{t(q-t)}\, dt \ .$$
Assume that $F(\mu, \e ,q)=0$ for some $\mu > 0$. Subtracting $F(\mu, \e ,q)\log q $ from the above expression for $\partial_\mu F (\mu, \e ,q)$, 
we infer from \eqref{alterF} that
$$\frac{\partial F}{\partial \mu}(\mu, \e ,q)=\e q\int_0^q \frac{t^\e}{(q-t)^\e } \frac{ t^\mu (1-qt)^\e }{t(q-t)}\log \left (\frac qt \right )\, dt >0 \ .$$
This implies that $F(\cdot ,\e ,q)=0$ cannot have more than one zero in $\mathbb R_{>0}$. It means that $\lambda_0(u)$ is the only negative
eigenvalue of $L_u$ and hence $\lambda_1(u) \ge 0$. Since by \eqref{formula for lambda_n},
 $\lambda_n(u) = n - \sum_{k \ge n +1} \gamma_k(u)$ for any $n \ge 0,$
it follows that $0 \le \sum_{k=2}^\infty \gamma_k(u)\leq 1.$
For any $n \ge 0$, denote by $\tilde f_n(\cdot, u)$ the eigenfunction of $L_u$, corresponding to $\lambda_n(u)$, normalized by $\tilde f_n(q, u) = 1$.
By \eqref{f(q)}, it then follows that for any $n \ge 1$, $\langle \check f_n(\cdot, u) | 1 \rangle = \tilde f_n(0, u) \ne 0$
and hence by \eqref{vanishing gamma} that $\gamma_n(u) > 0.$
\end{proof}
In a next step, given an arbitrary potential $u \in \bigcap_{n \ge 1}H^{n}_{r,0}$, we want to express the Fourier coefficient $\widehat u (1) = \langle u | {\rm e}^{ix} \rangle $
in terms of the Birkhoff coordinates $\zeta_n(u)$, $n \ge 1$, of $u$.
Denote by $(f_p)_{p\ge 0}$ the orthonormal basis of eigenfunctions of $L_u$ with our standard normalization
$$
\la f_0\vert 1\ra >0\ ,\qquad  \la f_{n+1}\vert Sf_n\ra >0\ ,\quad  \forall  n \ge 0.
$$
Then we get
\begin{equation}\label{formula widehat u 1}
\widehat u (1)  = \sum_{p=0}^\infty \la u\vert f_p\ra \la f_p\vert {\rm e}^{ix}\ra 
=\sum_{p=0}^\infty -\lambda_p\la 1\vert f_p\ra \sum_{n=0}^\infty \la S^*f_p\vert f_n\ra \la f_n\vert 1\ra \, .
\end{equation}
The following lemma provides a formula for $\widehat u (1)$ in terms of the Birkhoff coordinates $\zeta_n(u)$, $n \ge 1$.
To keep the exposition as simple as possible we restrict to the case where $\gamma_n(u) > 0$ for any $n \ge 1$.
\begin{lemma}\label{ill3}
For any  $u \in \bigcap_{n \ge 1}H^{n}_{r,0}$ with $\gamma_n(u) > 0$ for any $n \ge 1$, 
\begin{equation}\label{c1}
\widehat u (1)  =-\sum_{n=0}^\infty \sqrt{\frac{\mu_{n+1}(u)\kappa_n(u)}{\kappa_{n+1}(u)}}\zeta_{n+1}(u)\overline {\zeta_n(u)}\ .
\end{equation}
\end{lemma}
\begin{proof}
Recall that in the case $\gamma_{n+1}(u) \ne 0$, the matrix coefficients $M_{np} :=  \la S^*f_p\vert f_n\ra$ are given by \cite[formulas (4.7), (4.9)]{GK}.
Using that $|\langle f_{n+1} | 1\rangle|^2   = \kappa_{n+1} \gamma_{n+1}$ and $\zeta_{n+1} = \langle 1 | f_{n+1} \rangle / \sqrt{\kappa_{n+1}}$, one then gets
$$
M_{np}=\sqrt{\frac{\mu_{n+1}}{\kappa_{n+1}}}\frac{\la f_p\vert 1\ra }{\lambda_p-\lambda_n-1}\zeta_{n+1}\ .
$$
Substituting this formula into the expression \eqref{formula widehat u 1} for $\widehat u (1) $ yields
$$
\widehat u (1)  =-\sum_{n=0}^\infty \sqrt{\frac{\mu_{n+1}\kappa_n}{\kappa_{n+1}}}\zeta_{n+1}\overline {\zeta_n}
\sum_{p=0}^\infty \frac{\lambda_p|\la 1\vert f_p\ra |^2}{\lambda_p-\lambda_n-1}\ ,
$$
where for convenience, $\zeta_0:=1$.
Since\footnote{We are grateful to Louise Gassot for drawing our attention to this cancellation.}
$$\sum_{p=0}^\infty \frac{\lambda_p|\la 1\vert f_p\ra |^2}{\lambda_p-\lambda_n-1}=\sum_{p=0}^\infty |\la 1\vert f_p\ra |^2+(\lambda_n+1)\mathcal H _{-\lambda_n-1}\ ,$$
and $\mathcal H _{-\lambda_n-1} = 0$ due to  $\gamma_{n+1} > 0$, 
formula \eqref{c1} follows.
\end{proof}

Let us now consider the sequence $\mathcal S_0(t, u^{(k)})$, $k \ge 1$, with $u^{(k)}$ given by \eqref{definition u^{(k)}}.
Since $\gamma_1(u^{(k)})> k $ (Lemma \ref{ill1}), $\gamma_n(u^{(k)})> 0$, $n \ge 2$, (Lemma \ref{ill2}), and since
 $\gamma_n$, $n \ge 1$, are conserved quantities of \eqref{BO}, it follows that 
 formula \eqref{c1} is valid for  $\mathcal S_0(t , u^{(k)})$ for any $t \in \R$  and $k \ge 1.$ Hence
\begin{equation}\label{definition xi_k}
 \xi_k(t):= \langle \mathcal S_0(t, u^{(k)})  | \ {\rm e}^{ix} \rangle 
 \end{equation}
 is given by
 $$
 \xi_k(t) =
 - \sum_{n=0}^\infty \sqrt{\frac{\mu_{n+1}(u^{(k)})\kappa_n(u^{(k)})}{\kappa_{n+1}(u^{(k)})}}
 \zeta_{n+1} ( \mathcal S_0(t, u^{(k)}))  \overline{\zeta_n( \mathcal S_0(t, u^{(k)}))}\ .
 $$
For any potential $u \in L^2_{r,0}$ and any $n\ge 1$, one has by \eqref{BOsolution}, \eqref{frequencies in Birkhoff}
$$
\zeta_n(\mathcal S_0(t, u))=\zeta_n(u){\rm e}^{it\omega_n(u)}\ ,\qquad 
\omega_n(u):=n^2-2\sum_{k=1}^\infty \min(k,n)\gamma_k(u)\ ,
$$
implying that $\omega_1(u) = 1 + 2 \lambda_0(u)$ and for any $n \ge 1,$
$$
\omega_{n+1}(u) - \omega_n(u) = 2n+1 -2 \sum_{k=n+1}^\infty \gamma_k(u) = 1 + 2 \lambda_n(u)
$$
so that, for any $n\ge 0$,
$$
\zeta_{n+1}(\mathcal S_0(t, u))\overline{\zeta_n(\mathcal S_0(t, u))}=\zeta_{n+1}(u)\overline{\zeta_n(u)}{\rm e}^{it(1+2\lambda_n(u))}\ .
$$
Therefore, we have
\begin{equation}\label{formula xi_k}
 \xi_k(t) =
 - \sum_{n=0}^\infty \sqrt{\frac{\mu_{n+1}(u^{(k)})\kappa_n(u^{(k)})}{\kappa_{n+1}(u^{(k)})}}\zeta_{n+1} (u^{(k)})\overline{\zeta_n(u^{(k)})}{\rm e}^{it(1 + 2\lambda_n(u^{(k)}))}\ .
 \end{equation}

\medskip

\noindent
{\em Proof of Theorem \ref{Theorem 2}.}
First we consider the case $c=0$.
Our starting point is formula \eqref{formula xi_k}.
Recall  the product formulas of \cite[Corollary 3.4]{GK}
$$
\kappa_0=\prod_{p \ge 1}\Big(1-\frac{\gamma_p}{\lambda_p-\lambda_0}\Big)\ ,\quad
 \kappa_n=\frac{1}{\lambda_n-\lambda_0}\prod_{1\leq p\ne n}\Big(1-\frac{\gamma_p}{\lambda_p-\lambda_n}\Big)\ ,\  \forall n\ge 1
 $$
 and (cf.  \cite[formula (4.9)]{GK})
$$
\frac{\mu_{n+1}}{\kappa_{n+1}}=\frac{\lambda_n+1-\lambda_0 }{\prod_{1\le p\ne n+1}   \Big(1-\frac{\gamma_p}{\lambda_p-\lambda_n-1}\Big) }\ .
$$
They yield
\begin{eqnarray*}
\frac{\mu_1\kappa_0}{\kappa_1}&=& \prod_{p \ge 1}\frac{1-\frac{\gamma_p}{\lambda_p-\lambda_0}}{1-\frac{\gamma_{p+1}}{\lambda_{p+1}-\lambda_0-1}}
=\prod_{p \ge 1} \Big(1-\Big(\frac{\gamma_p}{\lambda_p-\lambda_0}\Big)^2\Big)\ ,\\
 \frac{\mu_2\kappa_1}{\kappa_2}&=&\Big( 1+\frac{1}{\lambda_1-\lambda_0}  \Big)
 (1+\gamma_1 )^{-1}\prod_{p \ge 2}\frac{1-\frac{\gamma_p}{\lambda_p-\lambda_1}}{1-\frac{\gamma_{p+1}}{\lambda_{p+1}-\lambda_1-1}} \ , 
 \end{eqnarray*}
 and for any $n \ge 2$,
 $$
 \frac{\mu_{n+1}\kappa_n}{\kappa_{n+1}} = 
 \Big( 1+\frac{1}{\lambda_n-\lambda_0}  \Big)\frac{1+\frac{\gamma_1}{\lambda_n-\lambda_1}}{1+\frac{\gamma_1}{\lambda_n+1-\lambda_1}}\frac{\prod_{2\le p\ne n}
 \Big( 1-\frac{\gamma_p}{\lambda_p-\lambda_n} \Big)}{\prod_{1\le p\ne n}\Big( 1-\frac{\gamma_{p+1}}{\lambda_{p+1}-\lambda_n-1}   \Big)}\ .
$$
Since by the definition of $\gamma_1$, $\lambda_1(u^{(k)}) = \lambda_0(u^{(k)})  + 1 + \gamma_1(u^{(k)})$ 
and since by Lemma \ref{ill2}, $0 \le \lambda_1(u^{(k)}) \le 1$, it follows from Lemma \ref{ill1} that
$$
\gamma_1(u^{(k)})=\lambda_1(u^{(k)})  - ( \lambda_0(u^{(k)}) + 1)\sim -\lambda_0(u^{(k)}) \to +\infty 
$$
and 
$$
1-\Big(\frac{\gamma_1}{\lambda_1-\lambda_0}\Big)^2(u^{(k)}) 
= \big(1- \frac{\gamma_1}{\lambda_1-\lambda_0}\big)  \big(1+ \frac{\gamma_1}{\lambda_1-\lambda_0}\big) 
\sim \frac{2}{\gamma_1(u^{(k)})} \ .
$$
Furthermore by \eqref{estimgamma}, $\sup_{k \ge 1}\sum_{n\ge 2}\gamma_n(u^{(k)})\leq 1$.
One then concludes that
$$
\frac{\mu_1\kappa_0}{\kappa_1}(u^{(k)}) \sim  \frac 2{\gamma_1(u^{(k)})}\ ,\qquad
 \frac{\mu_2\kappa_1}{\kappa_2}(u^{(k)}) = O\Big( \frac{1}{\gamma_1(u^{(k)})   }\Big) \ ,
 $$
 and
 $$
 \frac{\mu_{n+1}\kappa_n}{\kappa_{n+1}}(u^{(k)}) =O(1)\ ,\quad \forall  n\ge 2.
$$
 From the above estimates, we infer that the sequence of functions $(\xi_k(\cdot))_{k \ge 1}$, given by \eqref{formula xi_k}, is uniformly bounded for $t\in \R$. 
 Assume that it converges pointwise to $0$ on an interval $I$ of positive length. Then by the dominated convergence theorem,
 $$
  \lim_{k \to \infty} \int_I\xi_k(t){\rm e}^{-it(1 + 2\lambda_0(u^{(k)}))}\, dt =  0\ .$$
 On the other hand, since
 $$
 \begin{aligned}
 \xi_k(t) & {\rm e}^{-it(1 + 2\lambda_0(u^{(k)}))} 
  =  -  \sqrt{\frac{\mu_{1}(u^{(k)})\kappa_0(u^{(k)})}{\kappa_{1}(u^{(k)})}}\zeta_{1} (u^{(k)}) \\
  & - \sum_{n\ge 1}^\infty \sqrt{\frac{\mu_{n+1}(u^{(k)})\kappa_n(u^{(k)})}{\kappa_{n+1}(u^{(k)})}}\zeta_{n+1} (u^{(k)})\overline{\zeta_n(u^{(k)})}
  {\rm e}^{it2(\lambda_n(u^{(k)}) - \lambda_0(u^{(k)}) )}
 \end{aligned}
 $$
 and $\lambda_n(u^{(k)}) - \lambda_0(u^{(k)}) \ge |\lambda_0(u^{(k)})|$ for any $n \ge 1$, the above estimates yield
 $$
 \int_I\xi_k(t){\rm e}^{-it(1 + 2\lambda_0(u^{(k)}))}\, dt=
  -\frac{\sqrt{2}\zeta_1(u^{(k)})}{\sqrt{\gamma_1(u^{(k)})}}\vert I\vert + O\Big(\frac{1}{|\lambda_0(   u^{(k)}    )|}\Big)
  $$
 which does not converge to $0$ as $k \to \infty$, since $|\zeta_1(u^{(k)})|=\sqrt{\gamma_1(u^{(k)})}$ and $|I| > 0$. 
 Hence the assumption that $(\xi_k(\cdot))_{k \ge 1}$ converges pointwise to $0$ on an interval $I$ of positive length leads
 to a contradiction and thus cannot be true. \\
Finally, let us treat the case where $c \in \R$ is arbitrary.
Consider the sequence $(u^{(k)} +c)_{k \ge 1}$ with $(u^{(k)})_{k \ge 1}$ given by \eqref{definition u^{(k)}}.
By Lemma \ref{ill1}, $\lim_{k \to \infty} (u^{(k)} +c) = c$ in $H^{-1/2}_{r,c}$.
Furthermore,  by \eqref{formula with c} - \eqref{solution map for BO with c} one has $\mathcal S(t, u^{(k)} + c) = \mathcal S_c(t, u^{(k)}) + c$
and by \eqref{frequ for c ne 0} - \eqref{evolution S_{c, B}},
$$
\zeta_n \big(\mathcal S_c(t, u^{(k)}) \big) = \zeta_n( u^{(k)}) {\rm e}^{it (2cn + \omega_n(u^{(k)}))} . 
$$
Hence with $ \xi_{k}(t)$ given by \eqref{definition xi_k}, one has 
$$
 \xi_{c, k}(t) := \langle \mathcal S(t, u^{(k)} + c)  | \ {\rm e}^{ix} \rangle =  {\rm e}^{i2ct }  \xi_{k}(t) \ , \qquad \forall k \ge 1\ .
$$
Since
 $$
 \int_I\xi_{c, k}(t) {\rm e}^{-i2ct}{\rm e}^{-it( 1 + 2\lambda_0(u^{(k)}))}\, dt 
 =  \int_I\xi_{k}(t) {\rm e}^{-it( 1 + 2\lambda_0(u^{(k)}))}\, dt \ ,
 $$
 Theorem \ref{Theorem 2} follows from the proof of the case $c=0$ treated above.
 \hfill $\square$

\medbreak

An immediate consequence of Theorem \ref{Theorem 2} is the following
\begin{corollary}\label{no extension for s=-1/2}
The Birkhoff map $\Phi$ does not continuously extend to a map $H^{-1/2}_{r,0} \to h^0_+$.
\end{corollary}


\appendix

\section{Restriction of $\Phi$ to $H^s_{r,0}$, $s> 0$}\label{restrictions}

The purpose of this appendix is to study the restriction of the Birkhoff map $\Phi$ to $H^s_{r,0}$ with $s > 0$ 
and to discuss applications to the flow map of the Benjamin--Ono equation.
\begin{proposition}\label{positive}
For any $s\ge 0$, the restriction of the Birkhoff map $\Phi $ to $H^s_{r,0}$
is a homeomorphism from $H^s_{r,0}$ onto $h^{s+1/2}_+$.
Furthermore, $\Phi : H^s_{r,0} \to h^{s+1/2}_+$ and its inverse $\Phi^{-1} :  h^{s+1/2}_+ \to H^s_{r,0} $
map bounded subsets to bounded subsets.
\end{proposition}
\begin{proof}
The case $s=0$ is proved in \cite{GK}. We first treat the case $s\in ]0,1]$. \\
Assume that $A$ is a bounded subset of $L^2_{r,0}$.
Given $u\in A$, let $K_{u; s} $, $0 \le s \le 1$, be the linear isomorphism of Lemma \ref{basis on Sobolev scale},
$$
f\in H^s_+ \to (\la f\vert f_n \ra)_{n\ge 0}\in h^{s}(\N_0)\ , \qquad f_n \equiv f_n(\cdot, u) \ .
$$
By Lemma \ref{basis on Sobolev scale}, the operator $K_{u; s}$ is 
uniformly bounded with respect to $u \in A$. Since $L_u 1 = -\Pi u$, one has
\begin{eqnarray}\label{K_u applied to u}
K_{u, s}(\Pi u)&=&\big( \la \Pi u\vert f_n\ra \big)_{n\ge 0} = \big( -\lambda_n(u)\la 1\vert f_n\ra \big)_{n\ge 0} \nonumber \\
&=&\big( -\lambda_n(u)\sqrt{\kappa_n(u)}\zeta_n(u) \big)_{n\ge 0}
\end{eqnarray}
where for convenience we have set $\zeta_0(u):=1$. Furthermore, by the proof of Proposition \ref{extension Phi, part 1}$(i)$ 
one has that
\begin{equation}\label{equivalent lambda kappa}
\lambda_n(u)= n+o\Big(\frac 1n \Big) \ ,\qquad  \sqrt{\kappa_n(u)}=\frac 1{\sqrt n} \big(1+o(1) \big)\ , 
\end{equation}
 uniformly with respect to $u \in A$.
Therefore, for any $s\in [0,1]$,  an element $u \in L^2_{r,0}$ belongs to  $H^s_{r,0}$ 
if and only if   $K_{u, 0}(\Pi u)\in h^s(\N_0)$ or, equivalently,  $(\sqrt n\zeta_n(u))_{n\ge 1}\in h^s_+$,  implying that $\Phi (u)\in h^{s+1/2}_+$. 
Furthermore, if $A$ is a bounded subset of $H^s_{r, 0}$, then $\Phi(A)$ is a bounded subset of $h^{s+1/2}_+$. 
Conversely, if $B$ is a bounded subset of $h^{s+1/2}_+$, a fortiori it is bounded in $h^{1/2}_+$ and hence $\Phi^{-1}(B)$ is bounded in $L^2_{r,0}$. 
As a consequence, the norm of $K_{u; s}^{-1}:h^s(\N_0)\to H^s_+$ is uniformly bounded with respect to $u \in \Phi^{-1}(B)$ and hence 
$\Phi^{-1}(B)$ is bounded in $H^s_{r,0}$.\\
 Next we prove that for any $0 < s \le 1$,
 $\Phi : H^s_{r,0} \to h^{s+1/2}_+$ and its inverse are continuous. Since $\Phi:L^2_{r,0} \to h^{1/2}_+$ is a homeomorphism, 
 we infer from Rellich's theorem and the boudedness properties of $\Phi$ and its inverse, derived above, 
 that for any $s\in (0,1]$,  $\Phi : H^s_{r,0} \to h^{s+1/2}_+$ and $\Phi^{-1} : h^{s+1/2}_+ \to H^s_{r,0}$ 
 are sequentially weakly continuous.  
 As in the proof of Proposition \ref{extension Phi, part 2}$(iii)$, to prove that $\Phi : H^s_{r,0} \to h^{s+1/2}_+$  is continuous,
  it then suffices to show that $\Phi$ maps any relatively compact subset $A$ of $H^s_{r,0}$ 
 to a relatively compact subset of $h^{s+1/2}_+$ . For $s=1$, this is straightforward. Indeed, given any  bounded subset $A$ of $H^1_{r,0}$, $A$  is relatively compact in
 $H^1_{r,0}$ if and only if $\{ D \Pi u \ | \ u \in A \}$ is a relatively compact in $L^2_+$. As $\{ T_u \Pi u \ | \ u \in A \}$  is bounded in $H^1_+$, this amounts to say that 
$\{ L_u(\Pi u) \ | \ u \in A \}$ is relatively compact subset in $L^2_+$. Since 
$$
\| L_u(\Pi u)\|^2=\sum_{n \ge 0}\lambda_n(u)^2|\la 1\vert f_n(\cdot, u)\ra |^2\ ,
$$
we infer from \eqref{equivalent lambda kappa} that $ \| L_u(\Pi u)\|^2 = \| \Phi (u)\|_{3/2}^2+ R(u)\ ,$ where $R$ is a weakly continuous functional 
on $H^1_{r,0}$. As a consequence,  $\{ L_u(\Pi u) \ | u \in A \}$ is relatively compact in  $L^2_+$ if and only if  $\Phi (A)$ is
relatively compact in $h^{3/2}_+$. 
This shows that $\Phi : H^1_{r,0} \to h^{3/2}_+$  is continuous. In a similar way one proves that $\Phi^{-1} : h^{3/2}_+ \to H^1_{r,0}$ 
is continuous. This completes the proof for $s=1$.

\smallskip

 To treat the case $s\in (0,1)$, we will use the following standard variant of Lemma \ref{rel compact subsets}.
\begin{lemma}
Let $s\in (0,1)$. A bounded subset $A_+ \subset H^s_+$ is relatively compact in $H^s_+$ if and only if for any $\e >0$, there exist
$N_\e \ge 1$ and $R_\e > 0$  so that for any $f \in A_+,$ the sequence $\xi_n:=\hat f(n), n\ge 0$ satisfies 
$$
\big(\sum_{n > N_\e} n^{2s} |\xi_n)|^2\big)^{1/2} 
< \e
\, , \qquad 
\big( \sum_{0 \le n \le N_\e} n^2|\xi_n|^2 \big)^{1/2} 
< R_\e \, .
$$
The latter conditions on $(\xi_n)_{n \ge 0}$ characterize relatively compact subsets of $h^{s}(\N_0).$
\end{lemma}
We then argue as in the proof of Proposition \ref{extension Phi, part 2}$(iii)$, using the operators $K_{u; s}$ and $K_{u; s}^{-1}$, 
to complete the proof  of Proposition \ref{positive} in the case $s\in ]0,1]$.

\smallskip

In order to deal with the case $s>1$, we need the following lemma.
\begin{lemma}\label{L-regularity}
Let $k$ be a nonnegative integer and assume that $A$ is a bounded subset of $H^k_{r,0}$. 
Then, for any $u \in A$ and $s\in [k,k+1]$, an element $f\in H^k_+$ is in $H^s_+$ if and  only if for any $j=0,\dots ,k$,  $L_u^jf\in H^{s-k}_+$, with bounds
which are uniform with respect to $u \in A$. 
Furthermore, $A$ is compact in $H^s_{r,0}$ if and only if for any $j=0, \dots ,k$,  $\{ L_u^j \Pi u\ | \ u \in A \}$ is compact in $H^{s-k}_+$.
\end{lemma}
\noindent
{\em Proof of Lemma \ref{L-regularity}.} The statement is trivial for $k=0$. Let us first prove it for $k=1$. 
Assume that $A$ is a bounded subset of $H^1_{r, 0}$ and write $s=t+1$ with $t \in [0,1]$. 
Then for any $f\in H^1_+$, one has $f\in H^s_+$ if and only if 
$f,$ $Df$ belongs to $H^t_+$ with uniform bounds and with correspondence of compact subsets. For any $u$ in $H^1_{r,0}$, 
the operator $T_u$ 
maps bounded subsets of $H^t_+$ to bounded subsets of $H^t_+$. Hence for any $f \in H^1_+,$
$f,$ $Df\in H^t_+$ if and only if
$f,$ $L_u f \in H^t_+$ with bounds which are uniform with respect to $u \in A$ and with correspondence of compact subsets. This completes the proof for $k=1$. 
For $k \ge 2$, we argue by induction. 
Assume $k\ge 2$ is such that the statement is true for $k-1$. Let $A$ be a bounded subset of $H^k_{r,0}$, and let $s \in [k,k+1]$.
Then for any $f \in H^k_+,$  one has
$f\in H^s_+$ if and only if $f,$ $Df \in H^{s-1}_+$ with uniform bounds and correspondence of compact subets. Since $s-1\ge k-1\ge 1$, $H^{s-1}_+$ is an algebra
and hence $T_u $ maps bounded subsets of $H^{s-1}_+$ to bounded subsets of $H^{s-1}_+$. Consequently, $f,$ $Df \in H^{s-1}_+$
if and only if $f,$ $L_uf\in H^{s-1}_+$, with bounds which are uniform with respect to $u \in A$ and with correspondence of compact subsets. 
By the induction hypothesis, this is equivalent to $L_u^j f \in H^{s-k}_+$, for $j=0, \dots , k$ with bounds which are uniform with respect to $u \in A$  
and with correspondence of compact subsets.
\hfill $\square$

\medbreak

Let us now come back to the proof of Proposition \ref{positive}. We start with the case where $s=k$ is a positive integer
and assume that $A$ is a bounded subset of $H^k_{r,0}$. Applying Lemma \ref{L-regularity}, 
one easily shows by induction on $k$ that for any $u \in A$ ,
 $L_u^j\Pi u \in L^2_+$ for $j=0, \dots ,k$  with bounds which are uniform with respect to $u \in A$
and with correspondence of compact subsets. In other words, for any $u \in A$, $\Pi u$ belongs to the domain of $L_u^k$, 
with bounds  which are uniform with respect to $u \in A$ and with correspondence of compact subsets. 
On the other hand, if $u$ belongs to a bounded subset of 
$L^2_{r,0}$, $K_{u,0} $ maps the domain of $L_u^k$ into the space of sequences $(\xi_n)_{n \ge 0}$ such that $(\lambda_n(u)^\ell \xi_n)_{n \ge 0}\in h^0(\N_0)$ 
for every $\ell=0, \dots , k$, hence into $h^{k}(\N_0)$, with bounds which are uniform with respect to $u \in A$ and with correspondence of compact subsets. 
Hence for any $u\in A$, $\Phi (u)\in h_+^{k+1/2}$ with bounds, which are uniform with respect to $u \in A$
and with correspondence of compact subsets.

Finally, let $s\in (k,k+1)$ and assume that $A$  is a bounded subset of $H^k_{r,0}$. 
It then follows by Lemma \ref{L-regularity} that for any $u \in A$,  
$L_u^j\Pi u\in H^{s-k}_+$ for $j=0,\dots ,k$ with bounds which are uniform with respect to $u \in A$ and with correspondence of compact subsets. 
Applying $K_{u; s-k}$ and Lemma \ref{basis on Sobolev scale}, this means 
that $K_{u; s-k}(L_u^j\Pi u)\in h^{s-k}(\N_0)$ for $j=0, \dots ,k$, with bounds which are uniform with respect to $u \in A$ and with correspondence of compact subsets. Since
$$
K_{u, 0}(L_u^jf)=(\lambda_n(u)^j \la f\vert f_n\ra )_{n\ge 0}
$$
this implies $K_{u; 0}(\Pi u)\in h^{s}(\N_0)$, and hence $\Phi (u)\in h^{s+\frac 12}_+$ with bounds which are uniform with respect to $u \in A$. 
The proof of the converse is similar, taking into account that  by the case $s=k$ treated above, for any bounded subset $B$ of  $h_+^{k+1/2}$, the set $\Phi^{-1}(B)$ 
is bounded in $H^k_{r,0}$. The correspondence of compact subsets is established similarly.
\end{proof}
Arguing as in the proofs of Theorems \ref{Theorem 1}, \ref{Theorem 3} and \ref{Theorem 4}, one deduces from Proposition \ref{positive} the following
\begin{corollary}\label{results for s positive}
Let $s > 0$ and $c \in \R$. Then for any $t\in \R $,  the flow map $ S^t=\mathcal S(t, \cdot)$  of the Benjamin--Ono equation leaves
the affine space $H^{s}_{r,c}$,  introduced in \eqref{affine spaces}, invariant.
Furthermore, there exists an integral 
$ I_s : H^{s}_r \to \R_{\ge 0}$ of \eqref{BO},
satisfying
$$
\| v \|_{s} \le  I_s(v)\, , \quad \forall v \in H^{s}_r \, .
$$
In particular, one has
$$
\sup_{t \in \R} \|\mathcal S(t, v_0)\|_{s}
\le  I_s(v_0)\ , \quad \forall v_0 \in H^{s}_r\, .
$$
In addition, for any $v_0\in H^s_{r, c}$, the solution $t\mapsto \mathcal S(t,v_0)$ is almost periodic in $H^s_{r, c}$, 
and the traveling wave solutions are orbitally stable in $H^s_{r, c}$. 
\end{corollary}
\begin{remark}
Corollary \ref{results for s positive} significantly improves \cite[Theorem 1.3]{IMOS}, which provides polynomial (in $t$) bounds
of solutions of \eqref{BO} in $H^s_r$ for $1/2 < s \le 1$.
\end{remark}

\section{Ill-posedness in $H^{s}_r$ for $s < -1/2$}\label{illposedness for s < -1/2}

The goal of this appendix is to construct for any $c \in \R$ a sequence  $(v^{(k)}(t, x))_{k \ge 1}$ of smooth solutions of \eqref{BO} so that
\begin{itemize}
\item[(i)]  $(v^{(k)}(0, \cdot))_{k \ge 1}$ has a limit in $H^{-s}_{r, c}$ for any $s>\frac 12$ and 
\item[(ii)]  for any $t\ne 0$,  $(v^{(k)}(t, \cdot))_{k \ge 1}$ diverges in the sense of distributions, even after renormalizing the flow by a translation in the spatial variable.
\end{itemize}
Since the same arguments work for any $c \in \R$, we consider only the case $c = 0$. 
In a first step, let us review the results from \cite{AH}, using the setup developed in this paper:  the authors of \cite{AH}
construct a sequence of one gap potentials $v^{(k)}(t, \cdot)$ , $k \ge 1$, of the Benjamin-Ono equation so that  $v^{(k)}(0, \cdot)$ converges
in $H^{-s}_{r, 0}$ for any $s > 1/2,$ whereas for any $t \ne 0,$ $(v^{(k)}(t, \cdot))_{k \ge 1}$ diverges in the sense of distributions.
 Without further reference, we use notations and results from \cite[Appendix B]{GK}, where one gap potentials have
been analyzed. Consider the following family
of one gap potentials of average zero,
$$
u_{0,q}(x)=2{\rm{Re}} \big( q {\rm e}^{ix}/(1-q{\rm e}^{ix})\big)\, , \quad
0 < q < 1\, .
$$
The gaps $\gamma_n(u_{0,q})$, $n \ge 1$, of $u_{0,q}$ can be computed as
$$
\gamma_{1, q} := \gamma_1(u_{0,q}) = q^2/(1-q^2) \, , \qquad
\gamma_n(u_{0,q}) = 0 \, , \quad \forall n \ge 2 \, . 
$$
The 
frequency $\omega_{1, q} := \omega_1 (u_{0,q})$ 
is given by (cf. \eqref{frequencies in Birkhoff})
$$
\omega_{1,q} = 1 -2\gamma_{1,q} = \frac{1-3q^2}{1 - q^2}.
$$
The one gap solution, also referred to as traveling wave solution, of the BO equation
with initial data $u_{0,q}$ is then given by
$$
u_q(t,x)=u_{0,q}(x+\omega_{1,q}t)\, , 
\quad \forall t \in \R\, .
$$
Note that for any $s > 1/2$, 
$$
\lim_{q \to 1}u_{0,q} = 
2 {\rm{Re}} \Big(\sum_{k=1}^{\infty} {\rm e}^{ikx}\Big)
= \delta_0 - 1
$$ 
strongly in $H^{-s}_{r,0}$
where $\delta_0$ denotes the periodic Dirac $\delta-$distribution, centered at $0$. 
Since $\omega_{1,q} \to - \infty$ as $q \to 1$,
it follows that for any $t \ne 0,$
$u_q(t, \cdot)$ diverges
in the sense of distributions as $q \to 1$.

\vskip0.25cm

Note that by the trace formula (cf. \eqref{trace formula}) $\| u_{0,q}\|^2 = 2 \gamma_{1,q}$ and hence
$\omega_{1,q} + \| u_{0,q}\|^2 = 1$, implying that 
a renormalization of the spatial variable  by $\eta_q t$ with $\eta_q = \| u_{0,q}\|^2 = 1-\omega_{1,q}$ removes the divergence, i.e., 
for any $s > 1/2$,
$u_{0,q}(x+ t)$ converges in $H^{-s}_{r,0}$ for any $t$ as $q \to 1$. 
In order to construct examples where no such renormalization is possible, we consider two--gap solutions.
To keep the exposition as simple as possible, let
$u$ be a two-gap potential with $\gamma_1>0, \gamma_2>0$ and hence $\gamma_n = 0$ for any $n \ge 3$.
From \cite[Section 7]{GK} we know that $\Pi u(z)=-z\frac{Q'(z)}{Q(z)}$ where
\begin{equation}\label{ISF}
Q(z)=\det (I-zM)\ , \qquad \forall z \in \C, \quad |z| < 1\ ,
\end{equation}
and $M :=  (M_{np})_{0\leq n,p\leq 1}$ is a $2 \times 2$ matrix with
$M_{np} =\sqrt{\frac{\mu_{n+1}\kappa_p}{\kappa_{n+1}}}\frac{\zeta_{n+1}\overline {\zeta_p}}{\lambda_p-\lambda_n-1}$ or
$$
M_{np} =\sqrt{\frac{(\lambda_n+1-\lambda_0) \kappa_p}{\prod_{1\le q\ne n+1}   \Big(1-\frac{\gamma_q}{\lambda_q-\lambda_n-1}\Big) }}
\frac{\zeta_{n+1}\overline {\zeta_p}}{\lambda_p-\lambda_n-1}\ .
$$
Here $\zeta_0:=1$ and  $\zeta_n=\sqrt{\gamma_n}\, {\rm e}^{i\varphi_n}$ ($1 \le n \le2$) whereas  $\zeta_n = 0$ for any $n\ge 3$.
Moreover, along the flow of the Benjamin-Ono equation, we have
$\dot \varphi_n=\omega_n$, $1 \le n \le 2$,
with 
$$\omega_1=1-2\gamma_1-2\gamma_2\ ,\qquad  \omega_2=4-2\gamma_1-4\gamma_2\ .$$
Let us express the entries of $M$ in terms of $\gamma_1, \gamma_2$ and 
$\varphi_1,\varphi_2$. First, we remark that
\begin{eqnarray*}
\kappa_0&=&\Big(1-\frac{\gamma_1}{\lambda_1-\lambda_0} \Big) \Big(1-\frac{\gamma_2}{\lambda_2-\lambda_0} \Big)=\frac{2+\gamma_1}{(1+\gamma_1)(2+\gamma_1+\gamma_2)}\ ,\\
\kappa_1&=&\frac{1}{\lambda_1-\lambda_0}\Big(1-\frac{\gamma_2}{\lambda_2-\lambda_1} \Big)=\frac{1}{(1+\gamma_1)(1+\gamma_2)}\ .
\end{eqnarray*}
Then the above formula for $M_{np}$ yields
$$
M_{00} = -\sqrt{\frac{\gamma_1(2+\gamma_1)(1+\gamma_1+\gamma_2)}{2+\gamma_1+\gamma_2 } }   \, \frac{{\rm e}^{i\varphi_1}}{1+\gamma_1}\ ,
\,\,\,\,
M_{01} = \sqrt{\frac{1+\gamma_1+\gamma_2}{1+\gamma_2}}\frac{1}{1+\gamma_1}\ ,
$$
$$
M_{10} = -\sqrt{ \frac{\gamma_2}{2+\gamma_1+\gamma_2} }\frac{{\rm e}^{i\varphi_2}}{1+\gamma_1}\ ,
\qquad
M_{11} = -\sqrt{ \frac{\gamma_1\gamma_2(2+\gamma_1)   }{ 1+\gamma_2   }  }\frac{{\rm e}^{i(\varphi_2-\varphi_1)}}{1+\gamma_1}\ .
$$
Consequently,
$Q(z)=1+\alpha z+\beta z^2$
with $\alpha = - \text{Tr} (M)$ and $\beta = \det(M)$ or
$$
\begin{aligned}
\alpha  & = \frac{\sqrt{\gamma_1(2+\gamma_1)}}{1+\gamma_1}\Big(\sqrt{\frac{1+\gamma_1+\gamma_2}{2+\gamma_1+\gamma_2 } } {\rm e}^{i\varphi_1}+     
\sqrt{ \frac{\gamma_2   }{ 1+\gamma_2   }  }{\rm e}^{i(\varphi_2-\varphi_1)}\Big)\ , \\
\beta  & =\sqrt{\frac{(1+\gamma_1+\gamma_2)\gamma_2  }{ (1+\gamma_2)(2+\gamma_1+\gamma_2) }    } \ {\rm e}^{i\varphi_2}\ .
\end{aligned}
$$
Taking the limit $\gamma_1\to \infty , \gamma_2\to \infty $, one sees that
$$\alpha ={\rm e}^{i\varphi_1}+{\rm e}^{i(\varphi_2-\varphi_1)}+o(1)\ ,\qquad  \beta ={\rm e}^{i\varphi_2}+o(1)\ ,$$
which means that $Q(z)=(1+q_1z)(1+q_2z)$ where
$$q_1={\rm e}^{i\varphi_1}+o(1)\ ,\qquad  q_2={\rm e}^{i(\varphi_2-\varphi_1)}+o(1)\ .$$
Since $|q_1|<1\ ,\ |q_2|<1$, the family $u$ is  bounded in $H^{-s}_{r, 0}$ for any $s>\frac 12$.

\medskip

We define our sequence $(v_0^{(k)})_{k \ge 1}$ of initial data as the above two-gap potential 
$u$ with $\varphi_1(0)=\varphi_2(0)=0$, and $\gamma_1=\gamma_2 :=\gamma^{(k)}$
where $\gamma^{(k)} \to \infty $. 
At $t=0$, we infer that for any $z$ in the unit disc  
$$
\lim_{k \to \infty}\Pi v_0^{(k)}(z) =  \frac{2z}{1+z}
$$
and hence $v_0^{(k)}\to 2(1-\delta_{\pi})$ in $H^{-s}_{r,0}$ for any $s>\frac 12$. For $t\ne 0$,
we have
$$\varphi_1(t)=t(1-4\gamma ^{(k)})\ ,\qquad \varphi_2(t)-\varphi_1(t)=t(3-2\gamma ^{(k)})$$
and therefore
$$\Pi v^{(k)}(z,t)=\frac{ z{\rm e}^{i\varphi_1(t)} }{1+ z{\rm e}^{i\varphi_1(t)}  }+
\frac{ z{\rm e}^{i(\varphi_2-\varphi_1)(t)} }{1+ z{\rm e}^{i(\varphi_2-\varphi_1)(t)}  }+o(1)\ .$$
We thus conclude that $(\Pi v^{(k)}(z,t))_{k \ge 1}$ does not have a limit if $t\ne 0$ for any $z$ with $0<|z|<1$, even after renormalizing $z$ 
by a phase factor ${\rm e}^{i\eta_k(t)}$ with $\eta_k(t)$  a function of our choice. This proves that for $t\ne 0$, 
$(v^{(k)}(t, \cdot))_{k \ge 1}$ has no limit in $\mathcal D'(\T )$, even after renormalizing it by a $t-$dependent translation.
In particular, the renormalization of the spatial variable by $\| v_0^{(k)}\|^2  t=(4-\omega_2^{(k)})t$ does not
make $(v^{(k)}(t, \cdot))_{k \ge 1}$ convergent even in the sense of distributions.


\end{document}